\documentclass[art]{amsart}

\usepackage{amssymb,latexsym}
\usepackage{amsfonts}
\usepackage[pdftex]{hyperref}
\usepackage{color}
\usepackage[dvipsnames]{xcolor}
\usepackage{graphicx}
\usepackage[all]{xy}

\numberwithin{equation}{section}
\newtheorem{Thm}[equation]{Theorem}
\newtheorem*{THM}{Theorem}
\newtheorem{Lem}[equation]{Lemma}
\newtheorem{Cor}[equation]{Corollary}
\newtheorem{Prop}[equation]{Proposition}

\newtheorem{Conj}[equation]{Conjecture}

\theoremstyle{remark}
\newtheorem{Rem}[equation]{Remark}
\theoremstyle{remark}
\newtheorem{example}[equation]{Example}
\theoremstyle{definition}

\theoremstyle{definition}

\theoremstyle{definition}

\newtheorem{Def}[equation]{Definition}

\newcommand{\N}{\mathbb N}
\newcommand{\Z}{\mathbb Z}

\newcommand{\C}{\mathbb C}
\newcommand{\fg}{\mathfrak{g}}
\newcommand{\fh}{\mathfrak{h}}
\newcommand{\fb}{\mathfrak{b}}
\newcommand{\fl}{\mathfrak{l}}

\newcommand{\fu}{\mathfrak{u}}

\newcommand{\B}{\mathcal{B}}
\renewcommand{\N}{\mathcal{N}}

\newcommand{\X}{\mathcal{X}}

\DeclareMathOperator{\Ad}{Ad}

\DeclareMathOperator{\hgt}{ht}
\DeclareMathOperator{\ad}{ad}

\newcommand{\Cl}{O}
\newcommand{\R}{R}

\begin{document}

\title[The singular locus of semisimple Hessenberg varieties]{The singular locus of semisimple Hessenberg varieties}
\author{Erik Insko}
\address{Dept. of Mathematics, Florida Gulf Coast University, Fort Myers, FL 33965 }\email{einsko@fgcu.edu}
\author{Martha Precup}
\address{Dept. of Mathematics, Northwestern University, Evanston, IL 60208}\email{martha.precup@northwestern.edu}

%
%

%
\maketitle

\begin{abstract}  Although regular semisimple Hessenberg varieties are smooth and irreducible, semisimple Hessenberg varieties are not necessarily smooth in general. In this paper we determine the irreducible components of semisimple Hessenberg varieties corresponding to the standard Hessenberg space.  We prove that these irreducible components are smooth and give an explicit description of their intersections, which constitute the singular locus.  We conclude with an example of a semisimple Hessenberg variety corresponding to another Hessenberg space which is singular and irreducible, showing that results of this nature do not hold for all semisimple Hessenberg varieties.
\end{abstract}


\section{Introduction}

This paper initiates the study of the irreducible components and singular locus of semisimple Hessenberg varieties.  Our main results prove that semisimple Hessenberg varieties corresponding to the standard Hessenberg space (see Equation \eqref{eq:standard}) have smooth irreducible components. We also give an explicit description of these irreducible components and their intersections using the associated GKM graph.  

Hessenberg varieties are a collection of subvarieties of the full flag variety that generalize both Springer fibers and toric varieties associated to the Weyl chambers of the associated root system.  These varieties were first defined as subvarieties of the flag variety by DeMari, Procesi, and Shayman in \cite{DMPS92}, and they appear in connection with the study of quantum cohomology of partial flag varieties \cite{Kos96,Rie03}, geometric representation theory \cite{Spr76,Pro90,Ste92,Tym08,Tef11}, numerical analysis \cite{DMS88}, and Schubert calculus \cite{AT10, Ins15, IT16, HT11, Dre15}.  

Let $G$ be a linear, reductive algebraic group over $\C$, $B$ be a Borel subgroup, and $\B = G/B$ denote the corresponding flag variety. As usual, $\mathfrak g$ and $\mathfrak b$ denote the Lie algebras of $G$ and $B$ respectively, and $W$ is the Weyl group of $G$. We define a  \emph{Hessenberg space} $H$ to be a subspace of $\fg$ that contains $\fb$ and is closed under the Lie bracket with any element of $\fb$. In this paper, we focus our attention on the \emph{standard Hessenberg space} given by 
\begin{equation} H_{\Delta}:= \fb \oplus \bigoplus_{\alpha\in \Delta} \fg_{-\alpha}  \label{eq:standard} \end{equation} where $\Delta$ denotes the simple roots of $\fg$ and $\fg_{-\alpha}$ is the root space associated  to the negative simple root $-\alpha \in -\Delta$.

Given an element $X \in \mathfrak g$ and a fixed Hessenberg space $H$, the \emph{Hessenberg variety} $\B(X,H)$ is the subvariety of $\B$ consisting of all cosets $gB$ such that $Ad(g^{-1})(X)$ is an element of $H$.  We say that $\B(X,H)$ is semisimple when $X$ is a semisimple element of $\fg$ and that $\B(X,H)$ is nilpotent when $X$ is a nilpotent element of $\fg$.  Similarly, if $X\in \fg$ is a regular element we call the Hessenberg variety $\B(X,H)$ a regular Hessenberg variety.  As another example, when $X\in \fg$ is nilpotent and $H=\fb$ the nilpotent Hessenberg variety $\B(X,\fb)$ is the Springer fiber of $X$.  

Since DeMari, Procesi, and Shayman initiated their study of these varieties \cite{DMPS92}, the geometric and topological properties of Hessenberg varieties have received considerable attention leading to fruitful and surprising discoveries.  Several authors have proven Hessenberg varieties are paved by affines for increasingly general classes of Hessenberg varieties \cite{DMPS92,Tym06,Pre13}. These pavings show that some Hessenberg varieties are equivariantly formal \cite{GKM98,Tym05} and yield methods for computing their Betti numbers and other topological invariants using combinatorial formulas.

If $X\in \fg$ is a regular nilpotent element, the regular nilpotent Hessenberg variety $\B(X,H_{\Delta})$ is called the Peterson variety.  Peterson and Kostant used Peterson varieties to construct the quantum cohomology of the flag variety \cite{Kos96}.  More recently, several authors have studied equations defining local coordinate patches to analyze singular loci and describing some Hessenberg varieties as local complete intersections \cite{IY12, ADGH16}, including the Peterson variety.  We discuss these methods in Sections \ref{sec: singular locus} and \ref{sec: examples} below.  Understanding these geometric properties helps identify possible obstructions to studying (equivariant) cohomology, intersection theory, K-theory, and Newton-Okounkov bodies for Hessenberg varieties.

 DeMari, Procesi, and Shayman showed that regular semisimple Hessenberg varieties are smooth and equidimensional for any Hessenberg space $H$.  They also showed that  the regular semisimple Hessenberg variety corresponding to the standard Hessenberg space is in fact the toric variety associated to the Weyl chambers \cite[Theorem 11]{DMPS92}.  The Weyl group action on the cohomology of this variety had been studied independently by Procesi and Stembridge \cite{Pro90,Ste92}.  There is an action of the Weyl group on the cohomology of regular semisimple Hessenberg varieties defined by Tymoczko \cite{Tym08} which generalizes the Weyl group action in the toric variety case.  This representation has gained recent notoriety due to a conjecture posed by Shareshian and Wachs in 2011 which was proved by Brosnan and Chow in \cite{BC15} and again by Guay-Paquet \cite{GP15} using different methods.

The geometry of regular semisimple Hessenberg varieties has received a great deal of attention in the literature, due in large part to the representation discussed in the previous paragraph.  However, there are only a few papers which consider non-regular Hessenberg varieties specifically (such as \cite{Pre15} and \cite{Tym06-2}).  In this manuscript we focus primarily on non-regular semisimple Hessenberg varieties corresponding to the standard Hessenberg space.  We prove that their geometry is determined by the combinatorics of the Weyl group and its cosets.  Our main result is as follows.

 \begin{THM}  Let $S\in \fg$ be a non-regular semisimple element of $\fg$ and $W_M$ be the Weyl group of the centralizer of $S$ in $G$. The semisimple Hessenberg variety $\B(S,H_{\Delta})$ is a union of irreducible components
 \[
 	\B(S,H_{\Delta}) = \bigcup_{v\in \mathcal{S}} \mathcal{X}_v
 \]
 where $\mathcal{S}\subseteq W$ is a subset of coset representatives for $W_M\backslash W$.  Each of the irreducible components is smooth so the singular locus of $\B(S,H_{\Delta})$ consists precisely of those points in the intersection of two or more irreducible components.
 \end{THM}
 
 The statement of this theorem is a combination of Theorems \ref{thm: irreducible components} and \ref{thm: sm-comp} below.  A precise description of the irreducible components $\X_v$ and the elements in $\mathcal{S}$ are given in Theorem~\ref{thm: irreducible components}.   Following the proof of Theorem~\ref{thm: sm-comp}, we describe the GKM graphs of the irreducible components and their intersections as subgraphs of the GKM graph of $\B(S,H_{\Delta})$.

As Tymoczko notes in \cite{Tym06-2}, ``It is usually difficult to identify the irreducible components of Hessenberg varieties.''  This makes our results all the more surprising.  In the same paper, Tymoczko poses the following questions.  
\begin{itemize}
\item (Question~5.2) Let $X$ be any linear operator. If the Hessenberg space $H$ is in banded form, is the Hessenberg variety $\B(X,H)$ necessarily pure-dimensional?
\item (Question~5.4) Are all semisimple Hessenberg varieties smooth?
\end{itemize}
 Our description of the irreducible components of non-regular semisimple Hessenberg varieties corresponding to the standard Hessenberg space (which is in banded Hessenberg form) shows that the answer to both of these questions is no.  
 
The rest of this article is structured as follows.  In the second section we provide a survey of notation and known results that will be used to prove our main theorems.  In Section \ref{sec: irreducible components} we prove Theorem~\ref{thm: irreducible components}, which describes the irreducible components of semisimple Hessenberg varieties corresponding to the standard Hessenberg space.  Section \ref{sec: singular locus} contains the proof Theorem~\ref{thm: sm-comp}, proving that these irreducible components are smooth varieties.  In Corollary~\ref{cor: irreducible components} we use Theorems \ref{thm: irreducible components} and \ref{thm: sm-comp} to calculate the singular locus of these varieties. The end of Section \ref{sec: singular locus} gives a description of the GKM  graph of the singular locus of $\B(S,H_{\Delta})$ as a subgraph of the GKM graph of $\B(S,H_{\Delta})$ and includes many examples.  While many of our arguments rely heavily on Lie-theoretic terminology, we hope that the reader who is more interested in studying the combinatorics and GKM theory of Hessenberg varieties will find this subsection visually appealing.

Finally in Section \ref{sec: examples} we describe how to use commutative algebra and computational software to analyze the geometry of semisimple Hessenberg varieties associated to any Hessenberg space for $G=GL_n(\C)$ using similar methods as \cite{ADGH16, IY12, WY12}. This approach is used to give an example of an irreducible semisimple Hessenberg variety that is singular, showing that the results of Sections \ref{sec: irreducible components} and \ref{sec: singular locus} need not hold for arbitrary semisimple Hessenberg varieties.

\subsection{Acknowledgements}  The authors thanks both Samuel Evens and Alexandar Woo for conversations which clarified and shortened some of the arguments below. 


\section{Background and Notation} \label{sec:background}

We now state results and definitions from the literature which will be used in later sections.  All algebraic groups in this paper are assumed be complex and linear.  Let $G$, $\fg$, and $\B$ be as in the introduction.

Let $B\subset G$ be a fixed Borel subgroup and $B_-$ denote the opposite Borel subgroup so that $T=B\cap B_-$ is a torus. Denote the Lie algebra of $B$ by $\fb$ and the Lie algebra of $T$ by $\fh$.  Write $U$ for the maximal unipotent subgroup of $B$ and let $\fu$ denote the Lie algebra of $U$.  Similarly, $U^-$ denotes the maximal unipotent subgroup of the opposite Borel.

Let $\Phi$ be the root system associated to $B$, with $\Phi^+$, $\Phi^-$, and $\Delta$ the subsets of positive, negative and simple roots in $\Phi$, respectively.  Denote the negative simple roots by $\Delta^-$.   Each positive root $\gamma\in \Phi^+$ can be written uniquely as $\gamma=\sum_{\alpha}n_{\alpha}\alpha$ for $n_{\alpha}\in \mathbb{Z}_{\geq0}$.  The \emph{height} of $\gamma$ is $\hgt(\gamma) = \sum_{\alpha\in \Delta}n_{\alpha}$.  Fix root vectors $E_{\gamma}$ in each root space $\fg_{\gamma}$ such that $\ad_S(E_{\gamma}) = [S,E_{\gamma}] =\gamma(S)E_{\gamma}$ for all $S\in \fh$ and
\begin{eqnarray}\label{eq: Lie bracket}
\ad_{E_{\gamma}}(E_{\beta}) = [E_{\gamma}, E_{\beta}] = \left\{  \begin{array}{ll} m_{\gamma,\beta}E_{\gamma+\beta} & \textup{if } \gamma+\beta\in \Phi^+\\0 & \textup{otherwise} \end{array}\right.
\end{eqnarray}
for all $\gamma, \beta\in \Phi^+$ where $m_{\gamma, \beta}$ is a nonzero integer.  Let $U_{\gamma}=\exp(x_{\gamma}E_{\gamma})$ for $x_{\gamma}\in \C$ be the 1-dimensional unipotent subgroup corresponding to $\gamma\in \Phi$.
 
The Weyl group of $G$ is $W=N_G(T)/T$.  Throughout this paper, we fix a representative ${w}\in G$ such that $wU_{\gamma}w^{-1}=U_{w(\gamma)}$ for each $w\in W$ and use the same letter for both.  Write $s_{\gamma}$ for the reflection corresponding to $\gamma\in \Phi$.  The Weyl group of $G$ is generated by the simple refections $s_i = s_{\alpha_i}$ for each $\alpha_i\in \Delta$.   Given $w\in W$, the {\em length} of $w$ is the number of simple reflections in any reduced word $w=s_{i_1}s_{i_2}\cdots s_{i_{k}}$ for $w$, denoted by $\ell(w)$.

Our main example throughout this paper is the case in which $G=GL_n(\C)$ is the group of $n\times n$ invertible matrices and $\fg=\mathfrak{gl}_n(\C)$ is the collection of $n\times n$ matrices, also known as the type $A$ case.  In this setting, we take $B$ to be the subgroup of invertible upper triangular matrices, $B_-$ to be the opposite subgroup of invertible lower triangular matrices, and $T$ to be diagonal subgroup.  The Weyl group of $GL_n(\C)$ is the symmetric group $S_n$.

\subsection{Hessenberg varieties}
The main focus of this paper is a collection of subvarieties of $\B$ which we will now define.  

\begin{Def} A subspace $H\subseteq \fg$ is a {\em Hessenberg space} with respect to $\fb$ if $\fb \subseteq H$ and $[\fb, H]\subseteq H$.
\end{Def}

If $H\subset \fg$ is such a Hessenberg space then
\[
	H= \fb \oplus \bigoplus_{\gamma\in \Phi_H^-} \fg_{\gamma}
\]
for a subset $\Phi_H^-=\{ \gamma\in \Phi^-: \fg_{\gamma}\subseteq H \}$ of negative roots.  A Hessenberg variety is a subvariety of the flag variety defined as follows.  

\begin{Def}  Fix $X\in \fg$ and a Hessenberg space $H$ with respect to $\fb$.  
The {\em Hessenberg variety} associated to $X$ and $H$ is
\[
	\B(X,H)= \{ gB \in \B : g^{-1}\cdot X\in H \}
\]
where $g\cdot X$ denotes the adjoint action $Ad(g)(X)$.  
\end{Def}

In later sections we will specialize to the case in which $\Phi_H^- = \Delta^{-}$ and write $H_{\Delta}$ for this Hessenberg space, which we refer to as the \emph{standard Hessenberg space}.  

\begin{example}\label{ex: standard hess}  When $\fg = \mathfrak{gl}_n(\C)$, the standard Hessenberg space is the subspace of matrices
\[
	H=\left\{  \begin{bmatrix} a_{11} & a_{12} & a_{13} & \hdots & a_{1\,n-1}&  a_{1n}\\ a_{21} & a_{22} & a_{23} & \hdots & a_{2\,n-1} & a_{2n}\\ 0 & a_{32} & a_{33} & \hdots &  a_{3\,n-1} & a_{3n}\\ 0 & 0 & a_{43} & \hdots & a_{4\,n-1} & a_{4n}\\ \vdots & \vdots & \vdots & \ddots & \vdots& \vdots \\ 0 & 0 & 0 & \hdots & a_{n \,n-1} &a_{nn}  \end{bmatrix} : a_{i\,j}\in \C, 1\leq i \leq n, i-1\leq j \leq n  \right\}\subseteq \mathfrak{gl}_n(\C).
\]
\end{example}

When $S$ is a semisimple element of $\fg$, we let $M=Z_G(S)$ be the centralizer of $S$, so $M$ is a Levi subgroup of $G$, i.e. it is a closed reductive subgroup of $G$. $M$ acts on $\B(S,H)$ by translation.  Indeed, if $gB\in \B(S,H)$ then $mgB\in \B(S,H)$ for all $m\in M$ since 
\[
	(mg)^{-1}\cdot S= g^{-1}m^{-1}\cdot S = g^{-1}\cdot S\in H.
\]
If $\X\subseteq \B(S,H)$ is a subvariety, we denote the $M$-orbit of $\X$ by $M(\X)$.  Since conjugate elements of $\fg$ correspond to isomorphic Hessenberg varieties (see \cite[Remark 2.3]{Pre13}) we may assume without loss of generality that $M$ is a standard Levi subgroup of $G$.

\subsection{Cosets in the Weyl group}

Each standard Levi subgroup $L\subseteq G$ corresponds to a unique subset $\Delta_L\subseteq \Delta$ of simple roots.  We let $B_L=L\cap B$ denote the Borel subgroup of $L$ determined by $B$ and $U_{L}=L\cap U$ be the maximal unipotent subgroup of $B_L$ with Lie algebra $\fu_L$.  Similarly, let $U_L^-=L\cap U^-$ be the maximal unipotent subgroup of the opposite Borel, $L\cap B_-$. The Weyl group of $L$ is $W_L = \left< s_{\alpha} : \alpha\in \Delta_L \right>\subseteq W$.  Let $\Phi_L$ denote the root subsystem of $\Phi$ associated to $B_L\subset L$ with positive roots $\Phi_L^+$ and negative roots $\Phi_L^-$.  

\begin{example}  If we take $M=Z_G(S)$ as above, then $\Delta_M = \left< \alpha\in \Delta : \alpha(S)=0 \right>$ so the root system of $M=Z_G(S)$ is uniquely defined by the condition that $\gamma\in \Phi_M$ if and only if $\gamma(S)=0$.
\end{example}

For each $w\in W$, let 
\[
N(w) = \{ \gamma \in \Phi^+ : w (\gamma) \in \Phi^-\}.
\] 
We call $N(w)$ the \emph{inversion set} of $w$.  It is a well-known fact that $|N(w)|=|N(w^{-1})|=\ell(w)$.  We similarly define $N^-(w) = \{ \gamma\in \Phi^- : w(\gamma)\in \Phi^+ \}$.  If $\alpha\in N(w^{-1})\cap \Delta$, then $s_{\alpha}$ is called a \emph{left descent} of $w$.  Similarly, if $\alpha \in N(w)\cap \Delta$, then $s_{\alpha}$ is a \emph{right descent} of $w$.

\begin{Rem}\label{rem: descents}  
The elements of $N(w^{-1})$ can be characterized as follows: $\gamma\in N(w^{-1})$ if and only if $\ell(s_{\gamma}w) < \ell(w)$ (see \cite[Section 1.6]{Hum90}). 
\end{Rem}

Consider the sets
\[
	{{^L}W} = \{ v\in W : N(v^{-1}) \subseteq \Phi^+ \setminus \Phi_L^+ \}
\]
and
\[
	W^L = \{ v\in W: N(v) \subseteq \Phi^+ \setminus \Phi_L^+ \}.
\]
The elements of ${{^L}W}$ and $W^L$ form a set of shortest-length coset representatives for $W_L \backslash W$ and $W/W_L$ respectively in the following sense (see \cite[Proposition 2.4.4]{BB05}). 

\begin{Lem}\label{lem: w decomp}  
Each $w\in W$ can be written uniquely as 
\begin{itemize}
\item $w=yv$ for some $y\in W_L$ and $v\in {{^L}W}$, and
\item $w=v'y'$ for some $y'\in W_L$ and $v'\in {W^L}$
\end{itemize}
such that $\ell(w) = \ell(y) + \ell(v)= \ell(v')+\ell(y')$.
\end{Lem}

We will also use the following standard fact about inversion sets; especially in the context of the previous result. 

\begin{Lem}\label{lem: inversion set decomp} 
Let $w\in W$ and $w=yv$ for $y,v\in W$ such that $\ell(w)=\ell(y)+\ell(v)$.  Then $N(w^{-1}) = N(y^{-1}) \sqcup y N(v^{-1})$ and $N(w)=N(v)\sqcup v^{-1}N(y)$.
\end{Lem}

It's also a well known fact that $W_L$ must normalize $\Phi^+ \setminus \Phi_L^+$.  In particular 
\begin{eqnarray}\label{eq: W_L-norm}
y(\Phi^+ \setminus\Phi^+_{L}) \subseteq \Phi^+ \setminus \Phi^+_L \textup{ and } y(\Phi^- \setminus \Phi_L^- )\subseteq \Phi^- \setminus \Phi_L^-
\end{eqnarray} 
for all $y\in W_L$.

\subsection{Cellular decompositions}

The Bruhat decomposition of $G$ yields a corresponding cellular decomposition of the flag variety.  Namely, $\B=\bigsqcup_{w\in W} C_w$ where each $C_w = B wB /B$ denotes the \emph{Schubert cell} indexed by $w\in W$.  Each Schubert cell has the following explicit description:
\begin{eqnarray}\label{eq: Schubert cell}
C_w = U^w w B/B \textup{ where } U^w = \{ u\in U : w^{-1}uw \in U^- \} \cong \prod_{\gamma\in N(w^{-1})} U_{\gamma}.
\end{eqnarray}
Therefore $C_w \cong \C^{\ell(w)}$ and it is furthermore known that $\overline{C_w} = \bigsqcup_{y \leq w}C_y$ where $\leq$ denotes the Bruhat order on $W$.  We say that the affine cells $C_w$ \emph{pave} $\B$.

Let $S\in \fg$ be a semisimple element.  To begin our analysis of the semisimple Hessenberg variety $\B(S,H)$ consider
\begin{eqnarray*}\label{eqn: affine paving}
\B(S,H) = \bigsqcup_{w\in W} C_w\cap \B(S,H).
\end{eqnarray*}
We refer to the intersections on the right-hand side of the equation above as {\em Hessenberg-Schubert cells}.  The second author proved in \cite[Theorem 5.4]{Pre13} that each Hessenberg-Schubert cell is isomorphic to affine space and computes their dimension in \cite[Corollary 5.8]{Pre13}.    These results are summarized below.

\begin{Prop}\label{prop: cell dim} Suppose $S\in \fh$ is a semisimple element and $M=Z_G(S)$.  Given $w\in W$, write $w=yv$ with $y\in W_M$ and $v\in {^M W}$.  Then $C_w\cap \B(S,H)\cong \C^{d_w}$ where
\[
	d_w = |N(y^{-1})| + | N(v^{-1}) \cap v( \Phi_H^-) |.
\]
\end{Prop}

\begin{Rem}\label{rem: irreducible}
Since each Hessenberg-Schubert cell $C_w\cap \B(S,H)$ is isomorphic to affine space, the closure $\overline{C_w \cap \B(S,H)}$ is an irreducible subvariety of $\B(S,H)$.
\end{Rem}

\begin{example} \label{ex: regular ss dimn} If $S\in \fg$ is a regular semisimple element then $M=\{e\}$ so $W_M=\{e\}$ and ${^M W}=W$.  Proposition \ref{prop: cell dim} yields 
\[
	\dim(C_w\cap \B(S, H)) = | N(w^{-1}) \cap w(\Phi_H^-) | = |N^-(w) \cap \Phi_H^-| 
\]
for all $w\in W$, using the fact that $w^{-1}N(w^{-1}) = N^-(w)$. If $H$ is the standard Hessenberg space then $\dim(C_w\cap \B(S, H_{\Delta})) = |N^-(w) \cap \Delta^- | = |N(w)\cap \Delta|$ is the number of right descents of $w$.
\end{example} 

\subsection{Regular semisimple Hessenberg varieties}

In this paper we initiate the study of the singular locus of $\B(S,H)$.  When $S\in \fh$ is a regular semisimple element then $\B(S,H)$ is  a smooth variety \cite[Theorem 6]{DMPS92}.  

\begin{Prop}[DeMari-Procesi-Shayman] \label{prop: regular} Suppose $S \in \fh$ is a regular semisimple element and $H$ is a Hessenberg space in $\fg$ with respect to $\fb$.  Then $\B(S, H)$ is a smooth variety and $\dim(\B(S, H)) = |\Phi_H^-|$.  
\end{Prop}

In the same paper, DeMari-Procesi-Shayman also prove that the regular semisimple Hessenberg variety corresponding to the standard Hessenberg space is the toric variety associated with the Weyl chambers of the root system \cite[Theorem 11]{DMPS92}.  In general, one obtains the following corollary to the above proposition from \cite[Proposition A1]{AT10}.

\begin{Cor} \label{cor: reg-irreducible} Suppose $S \in \fh$ is a regular semisimple element and $H$ is a Hessenberg space in $\fg$ with respect to $\fb$.  If $\Delta^- \subseteq \Phi_H^-$ then $\B(S, H)$ is an irreducible variety such that
\[
	\B(S, H) = \overline{C_{w_0} \cap \B(S,H)}
\]
where $w_0\in W$ denotes the longest element of the Weyl group.
\end{Cor}

Below we consider the case in which $S\in \fh$ is not necessarily regular.  When the Hessenberg space is fixed as the standard one and there is no possible confusion, we write $\B(S,H_{\Delta}) = \B(S)$.


\section{Irreducible components} \label{sec: irreducible components}

Throughout this section and the next we assume that $H=H_{\Delta}$ is the standard Hessenberg space.  Let $S\in \fh$ denote a non-regular semisimple element.  In this section we identify the irreducible components of $\B(S)$.  Each one is the $M$-orbit of the closure of a certain Hessenberg-Schubert cell. We begin by associating each $v\in {W}$ to a subset of simple roots, $\R(v)$.  These subsets will be used later to identify which Hessenberg-Schubert cells correspond to irreducible components of $\B(S)$.  

\begin{Def}  For each $v\in W$, set $\R(v) :=  N(v) \cap \Delta$.  In other words, $\R(v)$ is the set of simple roots such that the simple reflections $s_{\alpha}$ for $\alpha\in \R(v)$ are right descents of $v$.  
\end{Def}

By Proposition \ref{prop: cell dim} when $H$ is the standard Hessenberg space and $w=yv$ with $y\in W_M$ and $v\in {^M W}$, then 
\begin{equation}\label{eqn:dim}
\dim(C_w\cap \B(S)) = |N(y^{-1})| + |N(v^{-1})\cap v(\Delta^-)| = |N(y^{-1})| + |\R(v)|. 
\end{equation}
In particular, if $v\in {^M W}$ then $\dim(C_v\cap \B(S))=|\R(v)|$.  

We now establish some notation for use in this section and the next.  For each $v\in {^M W}$, let $L\subseteq G$ denote the Levi subgroup corresponding to $\R(v)$ (in other words, $\Delta_L=R(v)$) with Lie algebra $\fl \subseteq \fg$ and associated flag variety $\B_L = L/B_L$. If there is any ambiguity, we write $L_v$ to indicate that $L_v$ is the Levi subgroup associated to $\R(v)$ for $v\in {^M W}$.

\begin{Rem} \label{rem: L-decomp} 
Let $w_v\in W_L$ denote the longest element of $W_L$.  Since $R(w_v) = R(v) \subseteq N(v)$, it follows from Lemma~\ref{lem: w decomp} that $v$ can be written uniquely as $v=x_vw_v$ for some $x_v\in W^{L}$ and $\ell(v) = \ell(x_v) + \ell(w_v)$.  
\end{Rem} 

\begin{example} \label{ex: 2-2}  Let $\fg=\mathfrak{gl}_4(\C)$ and $S = \text{diag}(1,1, -1, -1)$.  In this case, $M=GL_2(\C) \times GL_2(\C)$ and $W_M=\left< s_1, s_3 \right>$ is a Young subgroup of $S_n$.  The table below displays each element of ${^M W}$, the subset of simple roots $\R(v)$, and corresponding decomposition $v=x_vw_v$.
\begin{center}
$\begin{array}{c|c|c|c}
v\in {^M W} & \R(v) &  x_v & w_v \\ \hline
s_2s_3s_1s_2  & \{\alpha_2\} & {s_2s_3s_1} & {s_2}\\
s_2s_3s_1 & \{ \alpha_1,  \alpha_3 \} & {s_2} & {s_3s_1} \\
s_2s_1  &\{ \alpha_1\} & {s_2} & { s_1}\\
s_2s_3 &\{\alpha_3\} & {s_2} & {s_3}\\
s_2 &\{\alpha_2\} & e & {s_2} \\
e & \emptyset & e & e
\end{array}$ 
\end{center}
We will see below that this information can be used to characterize the closure relations among the Hessenberg-Schubert cells $C_v\cap \B(S)$ for $v\in {^M W}$.
\end{example}

Suppose $v\in {^M W}$ and $v=x_vw_v$ is the decomposition of $v$ given in Remark~\ref{rem: L-decomp}.  The next two lemmas establish some basic properties of this decomposition.  Since $w_v$ is the longest element of $W_L$, $N(w_v^{-1}) = \Phi_L^+$ and by Lemma~\ref{lem: inversion set decomp}, $N(v^{-1}) = N(x_v^{-1}) \sqcup x_v \Phi_L^+$.  

\begin{Lem}\label{lem: v-action} 
If $\gamma\in N(x_v^{-1})$, then $v^{-1}(\gamma) \in \Phi^- \setminus \Delta^-$.
\end{Lem}
\begin{proof}  Since $\gamma\in N(x_v^{-1}) \subseteq N(v^{-1})$ we certainly have that $v^{-1}(\gamma)\in \Phi^-$. Using Equation~\eqref{eq: W_L-norm} and the fact that $x_v\in W^L$ we get 
\[
	v^{-1}(\gamma) \in w_v^{-1}  x_v^{-1}N(x_v^{-1}) = w_v^{-1}N^-(x_v) \subseteq w_v^{-1}(\Phi^--\Phi_L^-) \subseteq \Phi^--\Phi_L^-.
\]
If $v^{-1}(\gamma)\in \Delta^-$, then there exists $\alpha\in \Delta$ such that $v^{-1}(\gamma)=-\alpha$ implying that $v^{-1}(\gamma) = -\alpha\in -R(v)\subseteq \Phi_L^-$, and contradicting the previous sentence.  Therefore $v^{-1}(\gamma)\in \Phi^- \setminus \Delta^-$. 
\end{proof}

\begin{Lem}\label{lem: cosets} 
$\tau = x_v z $ is an element of ${^M W}$ for all $z\in W_L$.
\end{Lem}  
\begin{proof}  This is a direct implication of Lemmas~\ref{lem: w decomp} and \ref{lem: inversion set decomp}.  Since $x_v \in W^{L}$ it follows that $\ell(x_v z) = \ell(x_v)+\ell(z)$ for all $z\in W_L$ by Lemma~\ref{lem: w decomp}.   As $w_v$ is the longest element of $W_L$ we know $N(z^{-1})\subseteq \Phi_L^+ = N(w_v^{-1})$.  Now by Lemma~\ref{lem: inversion set decomp},
\begin{eqnarray*}\label{eqn: inv set}
N(\tau^{-1}) = N(x_v^{-1}) \sqcup x_vN(z^{-1}) \subseteq N(x_v^{-1}) \sqcup x_v\Phi_L^+ = N(v^{-1}) \subseteq \Phi^+ - \Phi_M^+
\end{eqnarray*}
since $v\in {^M W}$.
\end{proof}

We define $\fu^v$ to be the Lie algebra of the unipotent subgroup $U^v$ defined in Equation~\eqref{eq: Schubert cell}, so
\[
	\fu^v := \bigoplus_{\gamma\in N(v^{-1})} \fg_{\gamma} = \bigoplus_{\gamma\in N(x_v^{-1})} \fg_{\gamma} \oplus \bigoplus_{\gamma\in x_v\Phi_L^+} \fg_{\gamma}.
\]
Our next result is a technical lemma which will be used to prove the proposition following it. 

\begin{Lem}\label{lem: technical step}
Let $X\in \fu^v$ and write $X = \sum_{\gamma\in N(v^{-1})} c_{\gamma}E_{\gamma}$ for some $c_{\gamma}\in \C$.  Suppose $\beta \in N(x_v^{-1})$ such that $\hgt(\beta) \geq k$ or $\beta\in x_v\Phi_L^+$.  If $c_{\gamma}=0$ for all $\gamma\in N(x_v^{-1})$ such that $\hgt(\gamma) < k$, then 
\[
	\ad_{-X}^{i}(E_{\beta}) \in \bigoplus_{\substack{\delta\in N(x_v^{-1})\\ \hgt(\delta)>k}} \fg_{\delta} \oplus \bigoplus_{\delta\in x_v\Phi_L^+} \fg_{\delta}
\]
for all $i\geq 1$.
\end{Lem}
\begin{proof}  Fix $i\geq 1$ and let $\ad^{i}_{-X}(E_{\beta}) = \sum_{\delta \in \Phi^+} d_{\delta}E_{\delta}$ for some $d_{\delta}\in \C$. Suppose $\delta \in \Phi^+$ such that $d_{\delta} \neq 0$, i.e., $E_{\delta}$ occurs as a summand of $\ad_{-X}^{i}(E_{\beta})$.  Since 
\[
X= \sum_{\substack{\gamma\in N(x_v^{-1})\\ \hgt(\gamma)\geq k}} c_{\gamma}E_{\gamma} + \sum_{\gamma\in x_v\Phi_L^+} c_{\gamma}E_{\gamma},
\]
it follows from the definition of the adjoint action in~\eqref{eq: Lie bracket} that
\begin{eqnarray*}
	\delta = \beta + \sum_{\substack{\gamma\in N(x_v^{-1})\\ \hgt(\gamma)\geq k}} n_{\gamma}\gamma + \sum_{\gamma \in x_v \Phi_L^+} n_{\gamma}\gamma
\end{eqnarray*}
for some $n_{\gamma}\in \Z_{\geq 0}$ such that $n_{\gamma}\neq 0$ for at least one $\gamma\in N(v^{-1})$ appearing in the index sets above because $i\geq 1$.  Note that $\delta \in N(v^{-1})$ since all roots in the equation above are elements of $N(v^{-1})$, and this set is closed under addition.

If $\beta\in N(x_v^{-1})$ such that $\hgt(\beta)\geq k$, then $\hgt(\delta)>\hgt(\beta)\geq k$.  If $\beta\in x_v\Phi_L^+$, we consider two possible cases.  Either $n_{\gamma}=0$ for all $\gamma\in N(x_v^{-1})$ or there exists at least one $\gamma\in N(x_v^{-1})$ such that $\hgt(\gamma)\geq k$ and $n_{\gamma}\neq 0$.  In the latter case, $\hgt(\delta)>\hgt(\gamma)\geq k$.  In the former case, $\delta \in x_v\Phi_L^+$ since $x_v\Phi_L^+$ is closed under addition (because $\Phi_L^+$ is).  We conclude that in every possible case, either $\hgt(\delta)> k$ or $\delta\in x_v\Phi_L^+$.  This proves the desired result since $N(v^{-1}) \setminus x_v\Phi_L^+ = N(x_v^{-1})$.
\end{proof}

The next proposition is a key step in the proof of Theorem \ref{thm: shortest coset closures} below.

\begin{Prop} \label{prop: key step}
If $uvB\in C_v\cap \B(S)$, then $u\in x_vU_L x_v^{-1}$.
\end{Prop}

\begin{proof}  Using the description of $C_v$ given in Equation~\eqref{eq: Schubert cell}, we begin by noting that if $uvB\in C_v$, then we may assume $u\in U^v = \prod_{\gamma\in N(v^{-1})} U_{\gamma}$.  Since $U^v$ is unipotent, the exponential map $\exp: \fu^v \to U^v$ is a diffeomorphism.  Therefore
\[
	u = \exp (X) \textup{ for some } X= \sum_{\gamma\in N(v^{-1})} c_{\gamma }E_{\gamma}
\]
with $c_{\gamma}\in \C$.  Recall that $N(v^{-1}) = N(x_v^{-1}) \sqcup x_v\Phi_L^+$.  To prove the proposition, it suffices to show that $c_{\gamma}=0$ for all $\gamma\in N(x_v^{-1})$.  Given this fact, $X= \sum_{\gamma \in x_v \Phi_L^+} c_{\gamma}E_{\gamma}$ so 
\[
	x_v^{-1}\cdot X = \sum_{\gamma\in x_v\Phi_L^+} c_{\gamma}E_{x_v^{-1}(\gamma)} \in \fu_L \Rightarrow \exp(x_v^{-1}\cdot X) \in U_L \Rightarrow x_v^{-1} \exp(X) x_{v} \in U_L
\]
which implies $u \in x_vU_L x_v^{-1}$ as desired.

We will prove $c_{\gamma}=0$ for $\gamma\in N(x_v^{-1})$ using induction on $\hgt(\gamma)$.  First we outline some additional notation and recall some facts about the adjoint action.  Since $u^{-1}\cdot S = \Ad(\exp(-X))(S) = \exp (\ad_{-X})(S)$ and $[S, E_{\gamma}] = \gamma(S)E_{\gamma}$ for all $\gamma\in \Phi^+$, we have
\begin{eqnarray*}
	u^{-1}\cdot S &=&  S + \sum_{i=1}^{\infty} \frac{1}{i!} \ad_{-X}^i(S) =S+ \ad_{-X}(S)+\sum_{i=2}^{\infty} \frac{1}{i!} \ad_{-X}^i(S) \\
	 &=& S + \left[ S\,, \sum_{\gamma\in N(v^{-1})} c_{\gamma} E_{\gamma} \right] + \sum_{i=2}^{\infty} \frac{1}{i!} \ad_{-X}^i(S)\\
	&=& S + \sum_{\gamma\in N(v^{-1})} c_{\gamma}\gamma(S)E_{\gamma} + \sum_{i=2}^{\infty} \frac{1}{i!} \ad_{-X}^i(S).
\end{eqnarray*}
Note that although the index set is infinite, the sum above is finite since $X\in \fu$ is a nilpotent element of $\fg$.  Consider 
\begin{eqnarray}\label{eq: ad-eqn} \begin{split}
	 \sum_{i=2}^{\infty} \frac{1}{i!} \ad_{-X}^{i}(S) &= \sum_{i=2}^{\infty} \ad_{-X}^{i-1}\left( \ad_{-X}(S) \right)\\ 
	 &=\sum_{i=2}^{\infty} \frac{1}{i!}\ad_{-X}^{i-1}\left(\sum_{\beta\in N(v^{-1})} c_{\beta}\,\beta(S)E_{\beta}  \right)\\
	 &= \sum_{i=2}^{\infty} \frac{1}{i!} \sum_{\beta\in N(v^{-1})} c_{\beta}\,\beta(S)\ad_{-X}^{i-1}(E_{\beta}). 
\end{split}\end{eqnarray}
Since $\fu^v$ is closed under the adjoint action, $\ad_{-X}^{i-1}(E_{\beta}) \in \fu^v$ for all $i\geq 2$ and $\beta\in N(v^{-1})$.  Thus we may write
\[
	\sum_{i=2}^{\infty} \frac{1}{i!} \ad_{-X}^i(S) = \sum_{\gamma\in N(v^{-1})} d_{\gamma}E_{\gamma}
\]
for some $d_{\gamma}\in \C$, so $u^{-1}\cdot S = S+\sum_{\gamma\in N(v^{-1})} (\gamma(S)c_{\gamma} + d_{\gamma})E_{\gamma}$.  Since $uvB\in \B(S)$ it must be the case that
\begin{eqnarray*}\label{eq: adjoint action}
	v^{-1}\cdot u^{-1}\cdot S = v^{-1}\cdot S + \sum_{\gamma\in N(v^{-1})} (\gamma(S)c_{\gamma} + d_{\gamma}) E_{v^{-1}(\gamma)} \in H_{\Delta} = \fb\oplus \bigoplus_{\alpha\in \Delta}\fg_{-\alpha}.
\end{eqnarray*}
In particular, if $v^{-1}(\gamma)\in \Phi^- \setminus \Delta^-$ then the above equation implies $\gamma(S)c_{\gamma} + d_{\gamma}=0$.  By Lemma~\ref{lem: v-action} we conclude that $\gamma(S)c_{\gamma} + d_{\gamma}=0$ for all $\gamma\in N(x_v^{-1})$.  To prove $c_{\gamma}=0$, we have only to show that $d_{\gamma}=0$ since $\gamma(S)\neq 0$ because $\gamma \in N(v^{-1}) \subseteq \Phi^+ -\Phi_M^+$.

It follows from Equations~\eqref{eq: Lie bracket} and \eqref{eq: ad-eqn} that  
\[
	\sum_{\gamma\in N(v^{-1})} d_{\gamma}E_{\gamma} =  \sum_{i=2}^{\infty} \frac{1}{i!} \ad_{-X}^i(S) \subseteq \bigoplus_{\substack{\delta \in \Phi^+ \\ \hgt(\delta )\geq 2}} \fg_{\delta}.
\]
This implies $d_{\gamma}=0$ for all $\gamma$ such that $\hgt(\gamma) =1$, implying $c_{\gamma}=0$ for all $\gamma$ such that $\hgt(\gamma)=1$, and proving the base case.

Now assume that $c_{\gamma}=0$ for all $\gamma\in N(x_v^{-1})$ such that $\hgt(\gamma) < k$.  Equation~\eqref{eq: ad-eqn} now becomes
\begin{eqnarray*}\label{eq: ad-eqn2}
\sum_{\gamma\in N(v^{-1})} d_{\gamma}E_{\gamma} =\sum_{i=2}^{\infty} \frac{1}{i!} \left( \sum_{\substack{\beta\in N(x_{v}^{-1}) \\ \hgt(\beta)\geq k}} c_{\beta}\,\beta(S)\ad_{-X}^{i-1}(E_{\beta}) + \sum_{\beta\in x_v\Phi_L^+} c_{\beta}\,\beta(S)\ad_{-X}^{i-1}(E_{\beta}) \right).
\end{eqnarray*}
Lemma~\ref{lem: technical step} implies that
\[
	\sum_{\gamma\in N(v^{-1})} d_{\gamma}E_{\gamma} \in \bigoplus_{\substack{\delta\in N(x_v^{-1})\\ \hgt(\delta)>k}} \fg_{\gamma} \oplus \bigoplus_{\delta\in x_v\Phi_L^+} \fg_{\delta}.
\]
Therefore $d_{\gamma}=0$ for all $\gamma\in N(x_v^{-1})$ such that $\hgt(\gamma)=k$.  This implies $c_{\gamma}=0$ for all $\gamma\in N(x_v^{-1})$ such that $\hgt(\gamma)=k$ and proves the inductive step.  We conclude $c_{\gamma}=0$ for all $\gamma \in N(x_v^{-1})$ as desired.
\end{proof}

Recall that if $L\subseteq G$ is a Levi subgroup then $\B_L = L/B_L$ denotes the corresponding flag variety.  Let $\iota_L: \B_L \hookrightarrow \B$ denote the inclusion of $\B_L$ into the flag variety $\B$ induced from the inclusion $L\subseteq G$.  In particular, if $uzB_L\in  \B_L$ for some $u\in U^{z}$ and $z\in W_L$, then $\iota_L(uzB_L) = uzB$.  Any subvariety of $\B_L$ can be viewed as a subvariety of $\B$ by considering its image under $\iota_L$.  Our next result shows that the closures of certain Hessenberg-Schubert cells are isomorphic to regular semisimple Hessenberg varieties in the flag variety of Levi subgroup.

\begin{Thm}\label{thm: shortest coset closures} 
For $v\in {^M W}$ let $L=L_v$ be the Levi subgroup corresponding to $R(v)\subseteq \Delta$ and $v=x_vw_v$ be the decomposition of $v$ given in Remark \ref{lem: cosets}. Set $S_v:= x_v^{-1}\cdot S\in \fh$.  Then $\overline{C_v\cap \B(S)}\cong \B_L(S_v)$ where $\B_L(S_v)$ is the regular semisimple Hessenberg variety in $\B_L$ corresponding to the standard Hessenberg space $H_{\R(v)} := H_{\Delta}\cap \fl \subset \fl$.
\end{Thm}

\begin{proof}  We can view any element of $\fh$ as a semisimple element of $\fl\subseteq \fg$ by restriction.  First we show that $S_v$ is a regular semisimple element of $\fl$, or equivalently that $\gamma(S_v)\neq 0$ for all $\gamma\in \Phi_L$.  If $\gamma\in \Phi_L^{+}$ then $x_v(\gamma)\in x_v\Phi_L^+ \subseteq N(v^{-1}) \subseteq \Phi^+ \setminus \Phi_M^+$ because $v\in {^M W}$. This implies $\gamma(S_v) = \gamma(x_v^{-1}\cdot S) = x_v(\gamma)(S)\neq 0$ so $S_v$ is indeed a regular element of $\fl$.

By Corollary~\ref{cor: reg-irreducible}, $\B_L(S_v) = \overline{C_{L,w_v} \cap \B_L(S_v)}$ where $C_{L,w_v}=B_L w_v B_L/B_L$ is a Schubert cell in $\B_L$.  Let $U_{L,w_v} = \{ u\in U_L : uw_vB_L \in \B_L(S_v) \}$ so $C_{L,w_v}\cap \B_L(S_v) = U_{L,w_v}w_vB_L/B_L$ using the description of Schubert cells given in Equation~\eqref{eq: Schubert cell}. Furthermore, $C_{L,w_v}\cap \B_L(S_v)$ can be viewed as an open subset in $\B$ by identifying it with its image $\iota_L(C_{L,w_v}\cap \B_L(S_v)) = U_{L,w_v}w_vB/B$.  We will prove  that $C_v\cap \B(S)$ is the $x_v$-translate of this image, namely $C_v\cap \B(S) =  x_v U_{L,w_v}w_vB/B$.  Our result then follows from the fact that translation respects closure relations in the flag variety so 
\[
	\overline{C_v\cap \B(S)} = x_v \overline{U_{L,w_v}w_vB/B} \cong \overline{C_{L,w_v}\cap \B_L(S_v)} = \B_L(S_v).
\]

Each element of $C_{L,w_v}\cap \B_L(S_v)$ is of the form $uw_vB_L$ for some $u\in U_{L,w_v}$.  Consider $\iota_L(uw_vB_L) = uw_vB$.  Since $uw_vB_L\in C_{L,w_v} \cap \B_L(S_v)$ we get
\begin{eqnarray*}
	w_v^{-1}u^{-1}\cdot S_v \in  H_{\R(v)} \Rightarrow w_v^{-1}u^{-1}x_v^{-1}\cdot S \in H_{\Delta} \cap \fl \subseteq H_{\Delta}
\end{eqnarray*}
so $x_vuw_vB \in C_v\cap \B(S)$, implying $x_vU_{L,w_v}w_vB/B \subseteq C_v\cap \B(S)$.  

If $uvB\in C_v\cap \B(S)$, Proposition \ref{prop: key step} implies that $u \in x_v U_L x_v^{-1}$.  Let $u' = x_v^{-1}  u x_v\in U_L$.  Our assumption that $uvB\in C_v\cap \B(S)$ implies  
\[
	w_v^{-1}(u')^{-1} x_v^{-1}\cdot S = v^{-1} u^{-1} \cdot S \in H_{\Delta} \Rightarrow w_v^{-1}(u')^{-1}\cdot S_v \in H_{\Delta}\cap \fl=H_{R(v)}
\]
since $u'w_v\in L$.  We conclude that $u'\in U_{L,w_v}$ and $uvB = x_v\iota_L(u'w_vB_L)$.  Thus $C_v\cap \B(S)\subseteq x_vU_{L,w_v}w_vB/B$ and $C_v\cap \B(S)= x_vU_{L,w_v}w_vB/B$.
\end{proof}

This theorem yields a cellular decomposition of $\overline{C_v\cap \B(S)}$  for each $v\in {^M W}$ using the Hessenberg-Schubert decomposition 
\[
\B_L(S_v) = \bigsqcup_{z\in W_L} C_{L,z}\cap \B_L(S_v)
\] 
where $C_{L,z}$ denotes the Schubert cell  $B_LzB_L/B_L$ in $\B_L$.  Let $U_{L,z}:= \{u \in U^z: u zB_L\in  \B_L(S_v) \}$.  It follows from the proof of Theorem~\ref{thm: shortest coset closures} that
\begin{eqnarray}\label{eq: description}
	\overline{C_v\cap \B(S)} = \bigsqcup_{z\in W_L} x_vU_{L,z} zB/B \subseteq \bigsqcup_{z\in W_L}C_{x_vz}.
\end{eqnarray}
since $C_{L,z}\cap \B_L(S_v) = U_{L,z}zB_L/B_L$ and $\iota_L(U_{L,z}zB_L/B_L) = U_{L,z}zB/B$.  In particular, Equation~\eqref{eq: description} and Lemma~\ref{lem: cosets} imply
\begin{eqnarray}\label{eq: description2}
\overline{C_v\cap \B(S)}\subseteq \bigsqcup_{\tau\in ^MW} C_{\tau}.
\end{eqnarray}

The next lemma partially characterizes the closure relations among Hessenberg-Schubert cells $C_v\cap \B(S)$ for $v\in {^M W}$.  Results of this nature are rare.  For example, a full understanding of the closure relations between the cells paving Springer fibers is unknown.  For each $v\in {^M W}$ define 
\[
\Cl(v):=\{ \tau\in {^M W} : \tau\neq v,\,\tau=x_vz \textup{ for some } z\in W_{L_v}, \textup{ and } R(\tau)=R(z) \}.
\]

\begin{Lem}\label{lem: closure relations}  
 If $v$ and  $\tau$ are distinct elements of ${^M W}$ then
\[
	C_{\tau}\cap \B(S) \subset \overline{C_v\cap \B(S)}
\]
if and only if $\tau\in \Cl(v)$.
\end{Lem}

\begin{proof}  If $v$ and $\tau$ are distinct elements of ${^M W}$ such that $C_{\tau} \cap \B(S)\subseteq \overline{C_v\cap \B(S)}$ then the description of $\overline{C_v\cap \B(S)}$ in~\eqref{eq: description} implies $C_{\tau} \cap \B(S) = x_v U_{L_v,z}zB/B$ for some $z\in W_L$, so $\tau=x_vz$ and $R(z) \subseteq R(\tau)$.  Furthermore, by Equation~\eqref{eqn:dim} we have
\[
	 |R(\tau)| = \dim(C_{\tau}\cap \B(S))   = \dim(x_v U_{L_v,z}zB/B) =\dim(C_{L,z}\cap \B_L(S_v)) = |R(z)|,
\]
so $R(z) = R(\tau)$ and $\tau\in \Cl(v)$.

Now we assume $\tau\in \Cl(v)$, so $\tau=x_v z$ for some $z\in W_{L_v}$ and $R(\tau) = R(v)$.  Given these assumptions, we have $x_v U_{L_v,z}zB/B \subseteq C_{\tau} \cap \B(S)$.  Our goal is to prove that this is an equality.  Consider the decomposition of $\tau$ defined in Remark \ref{rem: L-decomp}, namely $\tau = x_{\tau}w_{\tau}$.  Since $\Delta_{L_{\tau}} = R(\tau) = R(z) \subseteq \Delta_{L_v}$ it follows that $L_{\tau} \subseteq L_v$ and $U_{L_\tau} \subseteq U_{L_v}$.  Suppose $u\tau B \in C_{\tau}\cap \B(S)$.  By Proposition \ref{prop: key step}, if $u\tau B\in C_{\tau} \cap \B(S)$ then $u\in x_{\tau} U_{L_{\tau}} x_{\tau}^{-1} = \tau U_{L_\tau}^- \tau^{-1}$ because $w_{\tau} U_{L_\tau} w_{\tau}^{-1} = U_{L_{\tau}}^-$.  Therefore $x_v^{-1} u x_v \in z U_{L_{\tau}}^- z^{-1}$.  

Since $\Delta_{L_{\tau}} = R(\tau) = R(z)$ we know $\Delta_{L_\tau} \subseteq N(z)$ and therefore $\Phi_{L_{\tau}}^+ \subseteq N(z)$.  This together with the fact that $z\in W_{L_v}$ implies $z(\Phi_{L_{\tau}}^-) \subseteq \Phi_{L_v}^+$ so $z U_{L_{\tau}}^- z^{-1} \subseteq U^z$.  Let $u' = x_v^{-1} u x_v \in U^z$.  Our assumption that $u\tau B\in C_{\tau}\cap \B(S)$ implies
\[
	z^{-1} (u')^{-1} x_v^{-1} \cdot S = \tau^{-1} u^{-1} \cdot S \in H_{\Delta} \Rightarrow z^{-1} (u')^{-1}\cdot S_v \in H_{\Delta}\cap \fl = H_{R(v)}
\]
since $u'z \in L_v$.  Therefore $u\tau B = x_v u' w_v B \in x_vU_{L_v,z}zB/B$ and we conclude $C_{\tau}\cap \B(S)=x_vU_{L_v,z}zB/B \subseteq \overline{C_v\cap \B(S)}$.
\end{proof}


\begin{example}\label{ex: 2-2 part 2} Building on Example~\ref{ex: 2-2}, suppose that $\fg= \mathfrak{gl}_4(\C)$ and $S=\text{diag}(1,1,-1,-1)$. From the table in Example~\ref{ex: 2-2}, we see that if $v=s_2s_1s_3$ then $\Cl(v) = \{ s_2s_1, s_2s_3 \}$ and Lemma~\ref{lem: closure relations} implies
\[
	C_{s_2s_1}\cap \B(S),\; C_{s_2s_3}\cap \B(S) \subset \overline{C_{s_2s_3s_1} \cap \B(S)}.
\]
The decomposition in Equation \eqref{eq: description} becomes
\[
	\overline{C_{s_2s_3s_1} \cap \B(S)} = (C_{s_2s_3s_1} \cap \B(S)) \sqcup (C_{s_2s_1}\cap \B(S)) \sqcup (C_{s_2s_3}\cap \B(S)) \sqcup s_2B/B
\]
since $U_{L,e} = \{e\}$.  Note that $\Cl(v)$ predicts which cells $C_\tau \cap B(S) $ are completely contained in $
\overline{C_{v} \cap \B(S)}$, but there may also be portions of other Hessenberg-Schubert cells contained in 
$\overline{C_{v} \cap \B(S)}$. The set $\Cl(v)$ will be used in Theorem \ref{thm: irreducible components} to describe the 
irreducible components of $\B(S)$.  
   Similar calculations show
\[
	\overline{C_{s_2s_3s_1s_2}\cap \B(S)} = (C_{s_2s_3s_1s_2}\cap \B(S)) \sqcup s_2s_3s_1B
\]
and
\[
	\overline{C_{s_2}\cap \B(S)} = (C_{s_2}\cap \B(S)) \sqcup eB.
\]
\end{example}

 In order to identify the irreducible components of $\B(S)$ we will need the following lemma.

\begin{Lem}\cite[Lemma 39.2.1]{TY05} \label{lem: M-orbit closure}
Let $M$ be an algebraic group, $K$ a parabolic subgroup of $M$, $X$ a variety with an $M$-action, and $Y$ a $K$-stable closed subset of $X$.  Then the $M$ orbit $M(Y)$ is a closed subset of $X$.
\end{Lem}

Recall that $M$ acts on $\B(S)$, and since $M$ is a connected algebraic group, the irreducible components of $\B(S)$ are $M$-invariant \cite[\S8.2, Proposition (d)]{H75}.  In fact, we have the following characterization of the irreducible components of $\B(S)$.

\begin{Thm}\label{thm: irreducible components}  The irreducible components of $\B(S)$ are of the form 
\[
	\X_{v}:= M(\overline{C_{v}\cap \B(S)})
\] 
for $v\in\mathcal{S}:= {^M W} \setminus \left(\cup_{v\in {^M W}} \Cl(v)\right)$. In particular, $\X_{\tau} \subset \X_v$ if and only if $\tau \in \Cl(v)$.
\end{Thm}
\begin{proof}  Let $L=L_v$ be the standard Levi subgroup associated to $R(v)$.  First, $\overline{C_v\cap \B(S)}$ with $v\in {^M W}$ is a closed subvariety of $\B(S)$ which is clearly $B_M$-invariant since both $C_v$ and $\B(S)$ are $B_M$-invariant.  By Lemma~\ref{lem: M-orbit closure}, the $M$-orbit $\X_v := M(\overline{C_v\cap \B(S)})$ is a closed subvariety of $\B(S)$.  It must also be irreducible since $\overline{C_v\cap \B(S)}$ is irreducible.  

Next we have that
\[
	M(C_v\cap \B(S)) = M(C_v)\cap \B(S) = \bigsqcup_{y\in W_M} C_{yv}\cap \B(S).
\]
On the other hand, using the Bruhat decomposition for $M$, Equation~\eqref{eq: description} implies
\[
	M(C_v\cap \B(S)) = \bigsqcup_{y\in W_M}  U^yy x_vU_{L,w_v}w_vB/B
\] 
so $C_{yv}\cap \B(S) = U^yy x_vU_{L,w_v}w_vB/B$ for all $y\in W_M$.  The previous two sentences imply 
\[
	\B(S) = \bigsqcup_{v\in {^M W}} \bigsqcup_{y\in W_M} C_{yv}\cap \B(S) \subseteq \bigcup_{v\in {^M W}} \X_v
\]
so $\B(S) = \bigcup_{v\in{^M W}} \X_v$ is a decomposition of $\B(S)$ into irreducible components. This decomposition will be unique after we remove all $\X_{\tau}$ with $\tau\in {^M W}$ such that $\X_{\tau}\subseteq \X_v$. 

We now prove that $\X_{\tau} \subset \X_v$ if and only if $C_{\tau}\cap \B(S) \subset \overline{C_{v}\cap \B(S)}$.  It is clear that if $C_{\tau}\cap \B(S) \subset \overline{C_v\cap \B(S)}$ then $\X_{\tau} \subset \X_v$.  For the opposite direction, suppose $\X_{\tau}\subset \X_v$ and consider 
\[
gB \in C_{\tau}\cap \B(S)\subseteq \X_{\tau} \subset \X_v=M(\overline{C_v\cap \B(S)}).
\]  
By assumption, there exists $m\in M$ so that $mgB \in \overline{C_v\cap \B(S)}$.  The description of $\overline{C_v\cap \B(S)}$ given in Equation \eqref{eq: description2} implies $m\in B_M$. If not, then $m=b_1{y}b_2$ for some $b_1,b_2\in B_M$ and $e\neq{y}\in W_M$ so 
\[
mgB \in b_1{y}b_2(C_{\tau}\cap \B(S)) \subseteq C_{y\tau}. 
\] 
On the other hand, $C_{y\tau}\cap (\overline{C_v\cap \B(S)})=\emptyset$ by Equation~\eqref{eq: description2} since $y\tau \notin {^M W}$, so we obtain a contradiction.  Since $m\in B_M$ and $\overline{C_v\cap \B(S)}$ is $B_M$-invariant, we conclude $gB\in \overline{C_v\cap \B(S)}$.  The description of the set $\mathcal{S}$ and final assertion of the theorem now follows from Lemma~\ref{lem: closure relations}.
\end{proof}

As in the statement of Theorem~\ref{thm: irreducible components}, we adopt the notation $\X_v:=M (\overline{C_v\cap \B(S)})$.

\begin{example}\label{ex: 2-1-1a}  Let $\mathfrak{g} = \mathfrak{gl}_4(\C)$ and $S= \text{diag}(2, 2, -1, -3)$. In this case, $M=\left< s_1 \right>$ and ${^M W}$ contains 12 elements.  Fix $v=s_2s_3s_1s_2s_1$ so $R(v) = \{ s_1, s_2 \}$, $x_v = s_2s_3$, and $w_v=s_1s_2s_1\in W_L=\left<s_1,s_2\right>$.  In the table below, we consider the set of all $\tau=x_vz$ for $z\in W_L$.  We compute the simple roots $\R(\tau)$, and display the corresponding element $z\in W_L$, and $R(z)$.
\begin{center}
$\begin{array}{c|c|c|c}
\tau\in {^M W} : \tau=x_vz & z\in W_L &  R(\tau) & R(z)  \\ \hline
s_2s_3s_1s_2s_1 &  {s_1s_2s_1} & \{\alpha_1, \alpha_2\}  & \{\alpha_1, \alpha_2\}  \\
s_2s_3s_1s_2 & {s_1s_2} & \{ \alpha_2 \} & \{\alpha_2\} \\
s_2s_3s_2s_1 & {s_2s_1} &\{ \alpha_1, \alpha_3\} & \{ \alpha_1 \}  \\
s_2s_3s_1 & {s_1} &\{\alpha_1, \alpha_3\} & \{ \alpha_1 \}  \\
s_2s_3s_2 & s_2 &\{\alpha_2,\alpha_3\} & \{ \alpha_2 \}\\
s_2s_3 & e &  \{\alpha_3\} & \emptyset  \\
\end{array}$ 
\end{center}
From the table, we see that $\Cl(v)=\{s_2s_3s_1s_2\}$ so $\X_{s_2s_3s_1s_2} \subseteq \X_v$, but this is the only $\X_{\tau}$ such that $\X_{\tau}\subset \X_v$.  Doing similar computations for each $v\in {^M W}$ shows that 
\[
\mathcal{S} = \{ s_2s_3s_1s_2s_1, s_2s_3s_2s_1, s_2s_3s_1, s_2s_3s_2\}
\]
so $\B(S)$ has four irreducible components.
\end{example}

Our next two results give an explicit description of the cellular decomposition of each irreducible component $\X_v$ and describe the dimensions of these components and $\B(S)$ combinatorially.

\begin{Cor}\label{cor: explicit description}  Suppose $v\in \mathcal{S}$ and let $v=x_vw_v$ be the decomposition of $v$ defined in Remark \ref{rem: L-decomp}.  Then
\[
	\X_v = \bigsqcup_{y\in W_M} \bigsqcup_{z\in W_L} U^y y x_v U_{L,z} zB/B.
\]
Furthermore, $\X_v = \overline{C_{y_0v} \cap \B(S)}$ where $y_0\in W_M$ is the longest element.
\end{Cor}
\begin{proof}  Applying the Bruhat decomposition for $M$ and using the fact that each $\X_v$ is $B_M$-stable yields the decomposition above from the description of $\overline{C_v\cap \B(S)}$ in~\eqref{eq: description}.  From the proof of Theorem \ref{thm: irreducible components}, we know $C_{y_0v}\cap \B(S) = U^{y_0}y_0x_v U_{L,w_v}w_vB/B$.  The cellular decomposition given above shows that $C_{y_0v}\cap \B(S) \subset \X_v$  for all $v\in {^M W}$, so $\overline{C_{y_0\tau} \cap \B(S)} \subseteq \X_v$ for all $\tau\in \Cl(v)$ and $\overline{C_{y_0v}\cap \B(S)} \subseteq \X_v$.   Equality follows from Theorem~\ref{thm: irreducible components} since $\B(S) = \bigcup_{v\in {^M W}}\overline{C_{y_0v} \cap \B(S)}$ is another decomposition of $\B(S)$ into irreducible components.
\end{proof}

\begin{Cor}\label{cor: dimension}  For each $v\in \mathcal{S}$, $\dim(\X_v) = \ell(y_0) + |\R(v)|$ where $y_0$ denotes the longest element of $W_M$.  In particular, 
\[
\dim(\B(S)) = \ell(y_0) + \max_{v\in \mathcal{S}}\{|R(v)|\}.
\]
\end{Cor}
\begin{proof}  The dimension of $\B(S)$ will be the maximum dimension of its irreducible components.  Combining Corollary~\ref{cor: explicit description} and Equation \eqref{eqn:dim}, the dimension of each irreducible component is $\dim(\X_v) = \dim(C_{y_0v}\cap \B(S)) = \ell(y_0) + |\R(v)|$.
\end{proof}

\begin{example}\label{ex: 2-2 part 3}  Suppose $\fg= \mathfrak{gl}_4(\C)$ and $S=\text{diag}(1,1,-1,-1)$.  The table in Example~\ref{ex: 2-2} implies that the corresponding semisimple Hessenberg variety has three irreducible components
\[
	\B(S) = \X_{s_2s_1s_3s_2} \sqcup \X_{s_2s_1s_3} \sqcup \X_{s_2}
\]
since $\mathcal{S} = \{ s_2s_1s_3s_2, s_2s_1s_3, s_2 \}$.  By Corollary~\ref{cor: dimension},
\[
\dim(\B(S)) = \dim (\X_{s_2s_1s_3})=   |\Phi_M^+ | + |R(s_2s_1s_3)| = 2+2=4.
\]
Notice that $\dim(\B(S)) > \dim(C_{w_0}\cap \B(S)) = \dim(\X_{s_2s_1s_3s_2}) = 3$.  This shows that one cannot use the intersection of the Hessenberg variety $\B(S)$ and the big open cell $C_{w_0}$ to compute the dimension of $\B(S)$, which is always true in the regular case.
\end{example}


\section{The singular locus} \label{sec: singular locus}

We now prove that each of the irreducible components of $\B(S)$ described in the previous section is smooth.  Our main tool for investigating these components are patches.  In the first portion of this section, we use similar methods as Insko and Yong in \cite{IY12} to analyze the local properties of each irreducible component.  In the second, we describe the GKM graphs of each irreducible component and the singular locus of $\B(S)$ as subgraphs of the GKM graph of $\B(S)$.  Before going further, we make the following simplifying remark. 

\begin{Rem} \label{rem: simplifications} Any two points in the same $M$-orbit of $\X_v$ have isomorphic tangent spaces.  In addition, if the $T$-fixed points of $\X_v$ are smooth points then $\X_v$ is smooth \cite[Lemma 4]{DMPS92}.  Combining these statements, our strategy is to prove that $\tau B\in \X_v$ is a smooth point for all $\tau \in {^M W}$ such that $\tau B\in \X_v$.
\end{Rem}

\subsection{Patches}  The opposite big cell $B_-B/B \cong \C^{\ell(w_0)}$ provides an affine open neighborhood of $eB\in \B$ and we obtain an affine open neighborhood $\N_g:=gB_-B/B$ of each $gB\in \B$ by translation.

\begin{Def}  Let $\X$ be a subvariety of $\B$.  The affine open neighborhood, $\mathcal{N}_{g,\X}= \N_g \cap \X$ of $gB \in \X$ is called a {\em patch of $\X$ at the point $gB$}.  
\end{Def}

Given $g \in G$ such that $g B\in \X\subseteq G/B$, we now give explicit coordinates for $\mathcal{N}_{g, \X}$ as in as in \cite[\S 3]{IY12}. Consider the projection $\pi: G \to G/B$, and let $U^-$ denote the maximal unipotent subgroup of $B_-$.  Since $\pi$ is a trivial fibration over $B_-B/B$ with fiber $B$, it admits a local section
\[
	\sigma: B_-B/B \to G
\]  
such that $\sigma(B_-B/B) = U^- \subset G$.  Similarly, $\pi$ admits a local section
\[
	\sigma_{g}: g B_-B/B \to G
\]
defined by $\sigma_{g}=g\sigma g^{-1}$ so $\sigma_g(\N_g) = gU^-$.  This provides a scheme-theoretic isomorphism $\N_g \cong g U^-$.  The section $\sigma_{g}$ identifies explicit coordinates for $\N_{g, \X}$ by restricting $\pi$ and $\sigma_{g}$ to $\X$,
\begin{eqnarray}\label{eqn: patch description}
	\N_{g, \X} = \N_g \cap \X \cong \sigma_{g} (\N_g \cap \X) \cong \pi^{-1}(\X) \cap g U^-.
\end{eqnarray}
In particular, we view $\N_{g,\X}\cong \pi^{-1}(\X)\cap gU^-$ as a subscheme of $gU^-$.  The patch of $\X$ at $gB$ can be used to investigate the local structure of $\X$.  For example, the variety $\X$ is singular at $gB$ if and only if $\N_{g,\X}$ is singular at $g$.



For the case in which $\X=\X_v$ and $g=\tau\in {^M W}$, $\N_{\tau, \X_v}$ is the subscheme of $\tau U^-$ defined by the condition that $\tau u\in \N_{\tau, \X_v}$ for $u\in U^-$ if and only if $\tau uB \in \X_v$.  We make use of the fact that the unipotent subgroup $U^{-}\subset G$ can be factored as a product of root subgroups,
\begin{eqnarray}\label{eq: U-factor}
        U^{-}\cong U_{\gamma_1}U_{\gamma_2}\cdots U_{\gamma_r}
\end{eqnarray}
where $\Phi^-=\{ \gamma_1, \gamma_2, ..., \gamma_r \}$ is any ordering of the negative roots \cite[\S 28.1]{H75} in our proof below.

\begin{Thm}\label{thm: sm-comp}  The subvariety $\X_v\subseteq \B$ is smooth. 
\end{Thm}
\begin{proof}  Let $L=L_v$ be the Levi subgroup associated to the subset of simple roots $R(v)$ and write $v=x_vw_v$ for $w_v\in W_L$ the longest element and $x_v\in W^L$ as in the previous section.  By Corollary~\ref{cor: explicit description}, $\tau\in {^M W}$ satisfies $\tau B\in \X_v$ if and only if $\tau = x_v z$ for some $z\in W_{L}$. As in the proof of Theorem~\ref{thm: shortest coset closures}, let $S_v:= x_v^{-1}\cdot S\in \fh$, which is a regular semisimple element of $\mathfrak{l}$.  We will prove 
\[
\N_{\tau, \X_v} \cong U_M^- \times \N_{z, \B_{L}(S_v)}
\]  
where $\N_{z, \B_{L}(S_v)}$ is the patch of $\B_{L}(S_v)$ at $zB_L$.  Each element of $\N_{z, \B_{L}(S_v)}$ is of the form $zu$ for some $u\in U_L^-$ such that $zuB_L\in \B_L(S_v)$.

   
First we note that $\tau^{-1}U_M^-\tau \subseteq U^-$ and  $\tau^{-1}U_{M}^- \tau \cap U_{L}^- =\{e\}$.  The first statement is an obvious implication of the fact that $\tau\in {^M W}$.  To prove the second, it suffices to show that $\tau^{-1}(\Phi_M^-) \cap \Phi_L^- = \emptyset$.  If not, then there exists $\gamma\in \Phi_M^-$ and $\beta\in \Phi_{L}^-$ such that $\tau^{-1}(\gamma) = \beta$.  Now,
\[
	w_vz^{-1}x_v^{-1}(\gamma)=w_v\tau^{-1}(\gamma) =  w_v( \beta ) \in \Phi^+
\]
since $w_v(\Phi_L^-)\subseteq \Phi^+_{L}$ by definition of $w_v$ as the longest element of $W_L$.  On the other hand, $zw_v \in W_{L}$ so $x_v(zw_v)\in {^M W}$ by Lemma~\ref{lem: cosets} implying $w_vz^{-1}x_v^{-1}(\gamma) \in w_vz^{-1}x_v^{-1}(\Phi_M^-)\subseteq \Phi^-$, a contradiction.

Given $u_1\in U_M^-$ and $zu_2\in \N_{z, \B_{L}(S_v)}$, consider the product $u_1 \tau u_2$.  We claim that $u_1\tau u_2 \in \N_{\tau, \X_v}$.  Since $\tau^{-1}U_M^-\tau \subseteq U^-$ we have $u_1 \tau u_2\in \tau U^-$.  Furthermore, since $zu_2B_L \in \B_L(S_v)$, we get that $\tau u_2B = x_v \iota_L (zu_2B_L) \in \overline{C_v\cap \B(S)}$ by the proof of Theorem~\ref{thm: shortest coset closures}.  Therefore $u_1 \tau u_2B \in \X_v$, and our claim follows.  

Define the map
\[
	\phi: U_M^- \times \N_{z, \B_{L}(S_v)} \to \N_{\tau, \X_v}
\]
by $\phi(u_1, zu_2) = u_1 x_vz u_2  = u_1\tau u_2$ for all $u_1\in U_M^-$ and $z u_2 \in \N_{z, \B_{L}(S_v)}$.  The paragraph above shows that $\phi$ is well defined.  If $\phi(u_1,zu_2) = \phi(u_3, zu_4)$ for some $u_1,u_3\in U_M^-$ and $zu_2, zu_4 \in \N_{z, \B_{L}(S_v)}$, then
\[
	u_1 \tau u_2 = u_3 \tau u_4 \Rightarrow \tau^{-1}u_3^{-1}u_1\tau = u_4u_2^{-1} \in \tau^{-1}U_M^- \tau \cap U_L^-.
\]
Since $\tau^{-1}U_{M}^-\tau \cap U_L^- = \{e\}$, this forces $u_1=u_3$ and $u_2=u_4$ so $\phi$ is injective.  

Now suppose $\tau u\in \N_{\tau, \X_v}$, i.e. $u\in U^-$ such that $\tau u B\in \X_v$.  Using Equation~\eqref{eq: U-factor},
\[
	U^- \cong \tau^{-1}U_M^- \tau \times \prod_{\gamma\in \Phi^- :\, \tau(\gamma)\notin\Phi_M^-} U_{\gamma}
\]
so we may write $u=u_1u_2$ where $u_1\in \tau^{-1}U_M^- \tau$ and $u_2 \in \prod_{\gamma\in \Phi^- :\, \tau(\gamma)\notin\Phi_M^-} U_{\gamma}$.  Since $u_1'= \tau u_1 \tau^{-1} \in U_M^-$ and $\tau u_2 \tau^{-1} \in \prod_{\gamma\in \Phi \setminus \Phi_M} U_{\gamma}$, it follows that $\tau u B\in \X_v = M(\overline{C_v\cap \B(S)})$ if and only if $\tau u_2 B \in \overline{C_v\cap \B(S)}$.  Using the proof of Theorem~\ref{thm: shortest coset closures},
\[
	\tau u_2 B \in \overline{C_v\cap \B(S)} \Leftrightarrow zu_2B\in \overline{U_{L,w_v}w_vB/B} \Leftrightarrow zu_2B_L \in \B_L(S_v).
\]
The above equation makes sense only if $u_2\in U_L^-$ and $zu_2\in \N_{z,\B_L(S_v)}$.  Thus $(u_1' , zu_2)\in U_M^- \times \N_{z, \B_{L}(S_v)}$ such that $\phi(u_1' , zu_2) = \tau u$ and we conclude that $\phi$ is surjective.

Finally, since $\N_{z, \B_{L}(S_v)}$ is smooth at $z$ by Proposition \ref{prop: regular} and $U_M^- \cong \C^{\ell(y_0)}$ it follows that $N_{\tau, \X_v}$ is smooth at $\phi(e, z) = \tau$.
\end{proof}

We obtain the following corollary from Theorems~\ref{thm: irreducible components} and \ref{thm: sm-comp}.

\begin{Cor}\label{cor: irreducible components}  In the semisimple Hessenberg variety $\B(S)$, a nonempty intersection of any two irreducible components is singular, and the singular locus of $\B(S)$ is the union of all such intersections.   Furthermore, 
\begin{eqnarray*}
\X_v \cap \X_{\tau} = M(\overline{C_{v}\cap \B(S)}\, \cap \, \overline{C_{\tau}\cap \B(S)}).
\end{eqnarray*}
for all $v, \tau \in \mathcal{S}$.
\end{Cor}
\begin{proof}  The first part of the corollary is obvious from Theorems~\ref{thm: irreducible components} and \ref{thm: sm-comp}.  The second assertion follows from the equality
\[
	M(\overline{C_v\cap\B(S)})\cap M(\overline{C_{\tau}\cap \B(S)}) = M(\overline{C_{v}\cap \B(S)}\, \cap \, \overline{C_{\tau}\cap \B(S)}).
\]
The right-hand side of the above equation is clearly a subset of the left.  We have only to show that the left-hand side is also a subset of the right.  Let $gB\in M(\overline{C_v\cap\B(S)})\cap M(\overline{C_{\tau}\cap \B(S)})$, so there exists $m_1, m_2\in M$ such that $gB=m_1g_1B=m_2g_2B$ where $g_1B\in \overline{C_v\cap \B(S)}$ and $g_2B\in \overline{C_{\tau}\cap \B(S)}$.  Now,
\[
	m_1^{-1}m_2g_2B = g_1B\in \overline{C_{v}\cap \B(S)}.
\]
It follows that $m_1^{-1}m_2\in B_M$ from Equation~\eqref{eq: description2} using a similar argument as in the proof of Theorem~\ref{thm: irreducible components}.  Therefore, $g_1B \in m_1^{-1}m_2 \overline{C_{\tau}\cap \B(S)} = \overline{C_{\tau}\cap \B(S)}$ and $gB=m_1g_1B\in M(\overline{C_v\cap \B(S)} \cap \overline{C_{\tau}\cap \B(S)})$, proving our claim.
\end{proof}

The example below illustrates that regular semisimple Hessenberg varieties corresponding to the standard Hessenberg space can be singular and not pure-dimensional.  This answers Tymoczko's Questions 5.2 and 5.4 from \cite{Tym06} in the negative, and one can use Corollary~\ref{cor: irreducible components} to identify many examples of singular and non-pure-dimensional Hessenberg varieties. More examples are included in Subsection~\ref{subsection: GKM graphs}. 
\begin{example} \label{ex: 2-2 part 4} 
Consider $\B(S) \subset GL_4(\C)/B$ for $S=\text{diag}(1,1,-1,-1)$.  In Example~\ref{ex: 2-2 part 3} we saw that $\mathcal{S} = \{ s_2s_1s_3s_2, s_2s_1s_3, s_2 \}$ and a description of the closure $\overline{C_v\cap \B(S)}$ for each $v\in \mathcal{S}$ was given in Example~\ref{ex: 2-2 part 2}.  From these, we conclude that the singular locus of $\B(S)$ is
\[
	\X_{s_2s_1s_3s_2}\cap \X_{s_2s_1s_3} = M (s_2s_1s_3B) \;\text{  and  } \; \X_{s_2s_1s_3}\cap \X_{s_2} = M (s_2B)
\]
by Corollary \ref{cor: irreducible components}.  We saw in Example~\ref{ex: 2-2 part 3} that $\B(S)$ is not pure dimensional.
\end{example}


\subsection{GKM graphs}\label{subsection: GKM graphs}
Every semisimple Hessenberg variety $\B(S,H)$ has an action of the maximal torus $T$.  In fact, $\B(S,H)$ is a {\em GKM space} and there are combinatorial methods available for computing the $T$-equivariant cohomology of these varieties (see \cite{GKM98}; \cite{Tym05} provides an overview of GKM theory and examples of such computations).  Although our proofs above do not rely on GKM theory, it is frequently convenient to use the {\em GKM graph} (or moment graph) to help visualize examples. 

We now give a description for the GKM graph of $\B(S)$ and identify the subgraph associated to each irreducible component and the singular locus.  By definition, the GKM graph of a GKM space $\X$ has a vertex set of T-fixed points of $\X$ and edges connecting two fixed points if there exists a one-dimensional $T$-orbit whose closure contains these points.  The following definition summarizes this information for $\X=\B(S)$.

\begin{Def}\label{def: GKM graph}  
The {\em GKM graph} of $\B(S)$ has vertex set $W$ and directed edges
\[
\xymatrix{
w \ar[r]^{\gamma} & v
}
\] 	
for all $v,w\in W$ such that 
\begin{enumerate}
\item $w=s_{\gamma}v$ and $\ell(v)< \ell(w)$, and 
\item either $\gamma\in \Phi_M^+$ or $w^{-1}(\gamma) \in \Delta^-$.  
\end{enumerate}
We say that $w$ is the {\em source} of the edge $\xymatrix{w \ar[r]^{\gamma} & v}$ and $v$ is the {\em target}.
\end{Def}

In the definition above, we identify each $T$-fixed point $wB\in \B$ with $w\in W$.  Each edge $\xymatrix{w \ar[r]^{\gamma} & v}$ for $w,v\in W$ satisfying (1) corresponds to the one-dimensional $T$-orbit $U_{\gamma}wB\subset \B$ whose closure contains $wB$ and $vB$.  Every one-dimensional $T$-orbit in $\B$ is of this form \cite[Lemma 2.2]{CK03}.  Therefore the GKM graph of the full flag variety is the graph with vertex set $W$ and edges satisfying property (1).  

Recall that $\B(S)^T = \B^T$ \cite[Proposition 3]{DMPS92}. Thus the vertex set of the GKM-graph for $\B(S)$ is also $W$.  Given a 1-dimensional $T$-orbit $U_{\gamma}wB$, we have 
\[
u^{-1}\cdot S  = S + \gamma(S)x_{\gamma}E_{\gamma} \textup{ for all } u=\exp(x_{\gamma}E_{\gamma}) \in U_{\gamma} \textup{ where } x_{\gamma}\in \C.
\]
Therefore $u^{-1}\cdot S\in w\cdot H_{\Delta}$ if and only if $w^{-1}(\gamma) \in \Delta^-$ or $\gamma(S)=0$.  It follows that conditions (1) and (2) from Definition \ref{def: GKM graph} give precisely those one-dimensional $T$-orbits in $\B(S)$.  This confirms that the information given in Definition \ref{def: GKM graph} is correct.  

Let $w=yv$ with $y\in W_M$ and $v\in {^M W}$ be the decomposition given in Lemma~\ref{lem: w decomp}.  By Lemma~\ref{lem: inversion set decomp}, $N(w^{-1}) = N(y^{-1})\sqcup yN(v^{-1})$.  Any edge with source vertex $w$ has label $\gamma\in N(w^{-1})$ by condition (1) and Remark \ref{rem: descents}.  Condition (2) now implies that those labels are exactly the roots in the set 
\[
N(y^{-1})\sqcup (yN(v^{-1})\cap w(\Delta^-))= N(y^{-1})\sqcup (N(v^{-1}) \cap v(\Delta^-)),
\] 
so the number of edges with source vertex $w$ is $\dim(C_w\cap \B(S))$ by Proposition~\ref{prop: cell dim}. 

\begin{example}\label{ex: GKM graph of reg ss}  Consider the regular semisimple Hessenberg variety $\B(S)$ where $S\in \fh$ is a regular semisimple element.  In this case, $M=\{e\}$ so $\Phi_M^+=\emptyset$.  Our graph includes all edges with source vertex $w$, labeled by $\gamma \in N(w^{-1})$ such that $w^{-1}(\gamma)\in \Delta^-$.  Figure \ref{fig: reg ss GKM graph} below shows the GKM graphs of the variety $\B(S)$ in $GL_3(\C)/B$ and $SP_4(\C)/B$. 

\begin{figure}[h]
$\xymatrix{& s_1s_2s_1 \ar[rd]^{\alpha_1}\ar[ld]_{\alpha_2} & & && s_1s_2s_1s_2 \ar[rd]^{\alpha_1} \ar[ld]_{\alpha_2} &\\
s_1s_2 \ar[d]_{\alpha_1+\alpha_2} & & s_2s_1\ar[d]^{\alpha_1+\alpha_2} &&s_1s_2s_1 \ar[d]_{\alpha_1+\alpha_2}&& s_2s_1s_2 \ar[d]^{2\alpha_1+\alpha_2}\\
s_1 \ar[rd]_{\alpha_1} & & s_2 \ar[ld]^{\alpha_2}&&s_1s_2\ar[d]_{2\alpha_1+\alpha_2}&&s_2s_1\ar[d]^{\alpha_1+\alpha_2} \\
& e &&&s_1 \ar[rd]_{\alpha_1} &&s_2 \ar[ld]^{\alpha_2}\\
& &&&&e&
}$
\begin{caption}{The GKM graphs of the regular semisimple variety corresponding to the standard Hessenberg space in $GL_{3}(\C)/B$ (left) and in $SP_{4}(\C)/B$ (right). } 
\label{fig: reg ss GKM graph} 
\end{caption}
\end{figure}
\end{example}

\begin{Prop}\label{prop: GKM graph of Xv}  
Given $v\in \mathcal{S}$ let $L=L_v$ be the standard Levi subgroup corresponding to $R(v)\subseteq \Delta$ and $v=x_vw_v$ be the decomposition of $v$ such that $w_v\in W_L$ is the longest element and $x_v\in W^L$.  The GKM graph of the irreducible component $\X_v$ is the induced subgraph of the GKM graph of $\B(S)$ corresponding to the vertices \[V(\X_v):=\{ yx_vz : y\in W_M, z\in W_L  \}.\]
\end{Prop}
\begin{proof}  Our proof relies on the description of $\X_v$ given in Corollary~\ref{cor: explicit description}.  By this corollary, $wB\in \X_v$ if and only if  $w= yx_v z$ for some $y\in W_M$ and $z\in W_L$.  Thus $V(\X_v)$ as defined above is indeed the vertex set for the GKM graph of $\X_v$.  It remains to show that any edge in the GKM graph of $\B(S)$ between two vertices in $V(\X_v)$ corresponds to a one-dimensional $T$-orbit in $\X_v$.

Suppose $\gamma$ labels an edge between two vertices in $V(\X_v)$, with source  vertex $w=y\tau$ for some $y\in W_M$ and $\tau=x_v z\in {^M W}$ with $z\in W_L$.  Applying Lemma~\ref{lem: inversion set decomp} twice we get 
\[
N(w^{-1}) = N(y^{-1}) \sqcup yN(\tau^{-1}) = N(y^{-1}) \sqcup yN(x_v^{-1})\sqcup yx_v N(z^{-1}).
\]
Given $\gamma\in N(w^{-1})$, $\gamma$ must be an element of either $N(y^{-1})$, $yN(x_v^{-1})$, or $yx_v N(z^{-1})$.  We consider each of these three cases below.  

If $\gamma \in N(y^{-1})\subseteq \Phi_M^+$, then  $U_{\gamma} \subset U_M$ so $U_{\gamma} w B \subset \X_v$.  Now suppose that $\gamma \in yx_v N(z^{-1})$ and let $\beta = x_v^{-1}y^{-1}(\gamma)$ so $\beta \in N(z^{-1}) \subseteq \Phi_L^+$.  By condition (2) in Definition \ref{def: GKM graph} we have $z^{-1}(\beta) \in \Delta^-$ and since $z\in W_L$, $z^{-1}(\beta) \in \Delta^-\cap \Phi_L = \Delta_L^-$.  Thus $U_{\beta}zB_L \subseteq \B_L(S_v)$ so  $U_{\beta}\subseteq U_{L,z}$.  It follows that $U_{\gamma}wB = yx_vU_{\beta}zB \subset \X_v$ by Corollary~\ref{cor: explicit description}.

Finally, consider the case in which $\gamma \in yN(x_v^{-1})$.  We have that $s_{\gamma}w = y s_{y^{-1}(\gamma)} x_v z$.  Since $y^{-1}(\gamma)\in N(x_v^{-1})$, $\ell(s_{y^{-1}(\gamma)}x_v)<\ell(x_v)$ by Remark \ref{rem: descents}.  In particular, $s_{\gamma}w$ cannot be written in the form $y'x_v z'$ for some $y'\in W_M$ and $z'\in W_L$. Thus $s_{\gamma}w \notin V(\X_v)$ in this case, violating our assumption that the edge labeled by $\gamma$ has a target vertex in $V(\X_v)$.  We conclude that this case will never occur.
\end{proof}

We can now describe the GKM graph of the singular locus of $\B(S)$.

\begin{Cor}\label{cor: GKM graph of the singular locus}  
For all distinct elements $v, \tau \in \mathcal{S}$, define
\[
	V(v, \tau) = V(\X_v)\cap V(\X_{\tau})
\]
and let $\Gamma_{v,\tau}$ be the induced subgraph of the GKM graph of $\B(S)$ associated to the vertex set $V(v,\tau)$.  The GKM graph of the singular locus of $\B(S)$ is the union of subgraphs $\Gamma_{v,\tau}$ for all distinct $v,\tau \in\mathcal{S}$.
\end{Cor}

\begin{example} \label{ex: 2-2 part 5} Figure \ref{fig: figure 2-2 part 5} shows the GKM graph of $\B(S) \subset GL_4(\C)/B$ for $S=\text{diag}(1,1,-1,-1)$.  The GKM graph for each of the distinct irreducible components $\X_v$ for $v\in \mathcal{S} = \{ s_2s_1s_3s_2, s_2s_1s_3, s_2 \}$ is a different color; $\X_{s_2s_1s_3s_2}$ is red, $\X_{s_2s_1s_3}$ is blue, and $\X_{s_2}$ is yellow.  The GKM graphs of the intersections between these irreducible components are highlighted correspondingly: $\X_{s_2s_1s_3s_2}\cap \X_{s_2s_1s_3}=M(s_2s_1s_3B)$ is violet and $\X_{s_2s_1s_3}\cap \X_{s_2}=M(s_2B)$ is green.  Together, these form the GKM graph of the singular locus of $\B(S)$ which was calculated in Example~\ref{ex: 2-2 part 4}.  To make the graph easier to read, we have suppressed the specific label of each of the vertices and the edge labels, except for the vertices in ${^M W} = \{s_2s_1s_3s_2, s_2s_1s_3, s_2s_1, s_2s_3, s_2,e\}$.  

\begin{figure}[h]
$\xymatrix{&&& {\color{red}w_0} \ar@[red][dl] \ar@[red][d]\ar@[red][dr]&&&\\
& & {\color{RedViolet}\bullet}\ar@[blue][dl] \ar@[blue][d] \ar@[RedViolet][dr] \ar@[RedViolet][drr] & {\color{red}\bullet} \ar@[red][d] \ar@[red][drr] & {\color{red}\bullet} \ar@[red][d] \ar@[red][dr]&& \\
& {\color{blue}\bullet} \ar@[blue][dr]\ar@[blue][d] \ar@[blue][dl] & {\color{blue}\bullet} \ar@[blue][dll] \ar@[blue][dr]\ar@[blue][drr] & {\color{RedViolet}\bullet} \ar@[blue][dll]\ar@[blue][dl]\ar@[RedViolet][drr]& {\color{RedViolet}\bullet} \ar@[blue][dl] \ar@[blue][d] \ar@[RedViolet][dr] & {\color{red}s_2s_1s_3s_2} \ar@[red][d] & \\
 {\color{ForestGreen}\bullet } \ar@[YellowOrange][dr] \ar@[ForestGreen][drr] \ar@[ForestGreen][drrr]  & {\color{blue}\bullet} \ar@[blue][dr] \ar@[blue][drrr]& {\color{blue}\bullet} \ar@[blue][d] \ar@[blue][drr] & {\color{blue}\bullet} \ar@[blue][dl]\ar@[blue][drr]& {\color{blue}\bullet} \ar@[blue][dr] \ar@[blue][dl] & {\color{RedViolet}s_2s_1s_3} \ar@[blue][dl] \ar@[blue][d]&\\
& {\color{YellowOrange}\bullet} \ar@[YellowOrange][dr] \ar@[YellowOrange][drr] & {\color{ForestGreen}\bullet} \ar@[YellowOrange][d]\ar@[ForestGreen][drr] & {\color{ForestGreen}\bullet} \ar@[YellowOrange][d] \ar@[ForestGreen][dr] & {\color{blue}s_2s_1}\ar@[blue][d]&  {\color{blue}s_2s_3} \ar@[blue][dl]& \\
& & {\color{YellowOrange}\bullet}\ar@[YellowOrange][dr] & {\color{YellowOrange}\bullet}\ar@[YellowOrange][d] & {\color{ForestGreen} s_2} \ar@[YellowOrange][dl]&& \\
& &&{\color{YellowOrange}e}&&&
}
$
\begin{caption}{The GKM graph of $\B(S)\subseteq GL_4(\C)/B$ for $S=\text{diag}(1,1,-1,-1)$.}
 \label{fig: figure 2-2 part 5}
\end{caption}
\end{figure}
\end{example}

We close this section with two more examples.

\begin{example}\label{ex: 2-1-1b} Continuing Example~\ref{ex: 2-1-1a}, we consider the semisimple Hessenberg variety $\B(S) \subseteq GL_4(\C)/B$ corresponding to $S= \text{diag}(2, 2, -1, -3)$.  Instead of the full GKM graph of $\B(S)$, Figure \ref{fig: 2-1-1} displays the the GKM graph for $\bigsqcup_{v\in {^M W}} C_v\cap \B(S)$ (the induced subgraph corresponding to $^MW$). From Example~\ref{ex: 2-1-1a} we know
\[
\mathcal{S} = \{ s_2s_3s_1s_2s_1, s_3s_2s_3s_1, s_2s_3s_2, s_2s_3s_1  \}.
\]
Now Proposition \ref{prop: GKM graph of Xv} and Corollary \ref{cor: GKM graph of the singular locus} compute the GKM graphs of $\X_v$ for each $v\in \mathcal{S}$ and their intersections, respectively. The GKM graph of the intersections between the closures of the cells $\overline{C_v\cap \B(S)}$ for $v\in \mathcal{S}$ are highlighted in red in Figure \ref{fig: 2-1-1}.  This is the induced subgraph corresponding to the vertices $\bigcup_{v,\tau\in \mathcal{S}} \left(V(v,\tau)\cap {^M W}\right)$.  The singular locus of $\B(S)$ consists of the $M$-orbit of these intersections by Corollary~\ref{cor: irreducible components}.

\begin{figure}[h] 
$\xymatrix{& {\color{black}s_2s_3s_1s_2s_1} \ar@[black][dl]_{\alpha_3} \ar@[black][dr]^{\alpha_1+\alpha_2}&&&\\
{\color{black}s_2s_3s_1s_2}\ar@[black][d]_{\alpha_1+\alpha_2+\alpha_3}& & {\color{red}s_3s_2s_3s_1} \ar@[red][d]_{\alpha_1+\alpha_2+\alpha_3} \ar[dr]^{\alpha_2}&& \\
{\color{red}s_2s_3s_1} \ar[d]_{\alpha_2+\alpha_3} \ar@[red][dr]^{\alpha_1+\alpha_2} & & {\color{red}s_3s_2s_3} \ar@[red][dl]_{\alpha_3} \ar@[red][dr]^{\alpha_2}  & {\color{black}s_3s_2s_1} \ar@[black][d]^{\alpha_1+\alpha_2+\alpha_3} & \\
s_2s_1 \ar[dr]_{\alpha_1+\alpha_2}& {\color{red}s_2s_3} \ar@[red][d]^{\alpha_2+\alpha_3}  & & {\color{red}s_3s_2} \ar[d]^{\alpha_2+\alpha_3}\\
& {\color{red}s_2}\ar@[black][dr]^{\alpha_2} && s_3 \ar[dl]_{\alpha_3}& \\
& &{\color{black}e} && 
}
$
\begin{caption}{The induced subgraph corresponding to ${^M W}$ of the GKM graph of $\B(S)\subseteq GL_4(\C)/B$ for $S=\text{diag}(2,2,-1,-3)$.}
\label{fig: 2-1-1}
\end{caption}
\end{figure}

\end{example}

For our last example, we consider a (non-regular) semisimple Hessenberg variety in $SP_{4}(\C)/B$. 
  
\begin{example}\label{ex: 2-2 in SP} Let $S=\text{diag}(1,1,-1,-1)\in \fh\subset \mathfrak{sp}_4(\C)$, so $M=\left<  s_1 \right>$ and 
\[
	{^M W}=\{e, s_2, s_2s_1, s_2s_1s_2\}.
\]
In this case, $\mathcal{S}={^M W}-\{e\}$ since $\overline{C_v\cap \B(S)} \cong \mathbb{P}^1$ for all $v\in {^M W}$ such that $v\neq e$.  The point $eB = C_e\cap \B(S)$ is contained in $\overline{C_{s_2}\cap \B(S)}$.  Figure \ref{fig: 2-2 in SP} shows the GKM graph of each irreducible component of $\B(S)$ highlighted a different color; $\X_{s_2s_1s_2}$ is red, $\X_{s_2s_1}$ is blue, and $\X_{s_2}$ is yellow.  The GKM graphs of the intersections of these irreducible components are highlighted correspondingly: $\X_{s_2s_1s_2}\cap \X_{s_2s_1}$ is violet and $\X_{s_2s_1}\cap \X_{s_2}$ is green.  

\begin{figure}[h]
$\xymatrix{& {\color{red}s_1s_2s_1s_2} \ar@[red][rd]^{\alpha_1} \ar@[red][ld]_{\alpha_2} &\\
{\color{RedViolet}s_1s_2s_1} \ar@[RedViolet][drr]^{\alpha_1} \ar@[blue][d]_{\alpha_1+\alpha_2}&& {\color{red}s_2s_1s_2} \ar@[red][d]^{2\alpha_1+\alpha_2}\\
{\color{ForestGreen}s_1s_2} \ar@[ForestGreen][drr]^{\alpha_1} \ar@[YellowOrange][d]_{2\alpha_1+\alpha_2}&& {\color{RedViolet}s_2s_1}\ar@[blue][d]^{\alpha_1+\alpha_2} \\
{\color{YellowOrange}s_1} \ar@[YellowOrange][rd]_{\alpha_1} &&{\color{ForestGreen}s_2} \ar@[YellowOrange][ld]^{\alpha_2}\\
&{\color{YellowOrange}e}&
}$
\begin{caption}{The GKM graph of $\B(S)\subseteq SP_4(\C)/B$ for $S=\text{diag}(1,1,-1,-1)$.} 
\label{fig: 2-2 in SP} 
\end{caption}
\end{figure}
\end{example}


\section{Examples and Applications}  \label{sec: examples}
Many of the results proved in this paper began as conjectures formed using CoCalc (Sage) to analyze patch ideals in Type $A$ (when $G=GL_n(\C)$) following methods pioneered by Woo and Yong for Schubert varieties and Insko and Yong and Abe, Dedieu, Galetto, and Harada for regular nilpotent Hessenberg varieties \cite{ADGH16,IY12,WY08,WY12}. In this section we will:
\begin{itemize}
\item provide examples of such computations,
\item describe how to apply these computational techniques to study geometric properties of other semisimple Hessenberg varieties, and
\item provide an example which shows that the results of the previous two sections may fail for semisimple Hessenberg varieties that do not correspond to the standard Hessenberg space. 
\end{itemize}

In this section we fix the algebraic group $G=GL_n(\C)$.  Let $H\subseteq \mathfrak{gl}_n(\C)$ denote a Hessenberg space, and $S$ be a diagonal matrix in $\mathfrak{gl}_n(\C)$.  In this case, there exists a unique weakly increasing function $h: \{1,2,..., n\}\to \{1,2,..., n\}$ with $j\leq h(j)$ for all $1\leq j \leq n$ such that
\[
H= \{ A =[a_{ij}] \in \mathfrak{gl}_n(\C) : a_{ij}=0 \text{ for all } i>h(j) \}.
\]  
The equation above defines a bijective correspondence between Hessenberg spaces and all weakly increasing functions $h: \{1, 2, ..., n\}\to \{1,2,...,n\}$ such that $j\leq h(j)$ for all $1\leq j\leq n$.  We call any such function $h$ a {\em Hessenberg function} and denote it by $(h(1), h(2),..., h(n))$.  We will use this notation whenever it is convenient.

\begin{example}  When $H=H_{\Delta}$ is the standard Hessenberg space as in Example~\ref{ex: standard hess} the corresponding Hessenberg function is $h(i)=i+1$ for all $1\leq i \leq n-1$ and $h(n)=n$, or $(2,3,..., n-1, n,n)$.
\end{example}

Using the description given in Equation~\eqref{eqn: patch description}, the patch of $\B(S,H)$ at $wB$ is 
\[ \N_{w,\B(S,H)} \cong \left\{ wu \in wU^- : A= u^{-1}w^{-1} \cdot S\in H   \right\}.\]
Let $u$ denote a generic invertible lower-triangular unipotent matrix in $U^{-}$, and take the coordinates of $U^{-}$ to be
$\{x_{ij}: i>j \}$. Then $A= u^{-1}w^{-1} \cdot S \in H$ if and only if $a_{ij} =0$ for $i>h(j)$, where each $a_{ij}$ is a polynomial function in the $x_{ij}$-variables.  We define the ideal $I_{w,\B(S,H)} \subset \C[U^{-}]$ to be  $I_{w,\B(S,H)} := \langle a_{ij} : i > h(j)\rangle$.  If this ideal is radical, then $I_{w,\B(S,H)} = I(\N_{w,\B(S,H)})$ is called the \emph{patch ideal} for $\B(S,H)$ at $wB$.  

\begin{example} \label{ex: ideal calc} 
The standard Hessenberg space in $\mathfrak{gl}_4(\C)$ corresponds to the Hessenberg function $h=(2,3,4,4)$.  Let $S = \textup{diag}(1,1,-1,-1)$ and $w=s_2$.  The reader will recognize $\B(S)$ as the semisimple Hessenberg variety appearing in many of the examples from Section \ref{sec: irreducible components}.
To obtain the patch ideal $I_{s_2,\B(S)}$, we compute,
\begin{eqnarray*}
	A =  u^{-1}w^{-1} \cdot S &=& \begin{bmatrix} 1 & 0 & 0 &0 \\x_{21} & 1 & 0 & 0 \\  x_{31} & x_{32} & 1 & 0\\ x_{41} & x_{43} & x_{42} & 1 \end{bmatrix}^{-1} \begin{bmatrix} 1 & 0 & 0 & 0\\ 0 & -1 & 0 & 0\\ 0 & 0 & 1 & 0\\ 0 & 0 & 0 & -1 \end{bmatrix}\begin{bmatrix} 1 & 0 & 0 &0 \\x_{21} & 1 & 0 & 0 \\  x_{31} & x_{32} & 1 & 0\\ x_{41} & x_{43} & x_{42} & 1 \end{bmatrix}\\
	&=&\begin{bmatrix} 1 & 0 & 0 & 0\\ -2x_{21} & -1 & 0 & 0\\ 2x_{21}x_{32} & 2x_{32} & 1 & 0\\ -2x_{21}x_{32}x_{43}+2x_{21}x_{42}-2x_{41} &  -2x_{32}x_{43} & -2x_{43} & -1\end{bmatrix}
\end{eqnarray*}
and require that  
\begin{itemize}
\item $a_{31} = 2  x_{21} x_{32}=0$, 
\item $a_{42}= -2x_{32} x_{43}=0$, and
\item $a_{41} = -2x_{21}x_{32}x_{43}+2x_{21}x_{42}-2x_{41} =0$.
\end{itemize} 
These vanishing conditions define the ideal
\[
I_{w,\B(S)} = \langle g_{41}, g_{31}, g_{42}\rangle = \langle 
  x_{21}x_{42} - x_{41}, x_{21}x_{32},
  x_{32}x_{43} \rangle .
  \] 
In the equation above we have simplified the polynomial $a_{41}$ to the generator $g_{41}$ by subtracting a multiple of $a_{31}$ from $a_{41}$.  One can verify that $I_{s_2,\B(S)}$ is radical as these three generators form a square-free Gr{\"o}bner basis for $I_{s_2,\B(S)}$ \cite[Exercise 18.9]{Eis95}. Hence $I(\N_{s_2, \B(S)}) = I_{s_2, \B(S)}$ is the patch ideal for $\B(S)$ at $s_2B$.

Using the Jacobian criterion, we note that $s_2B$ is a singularity in the patch $\N_{s_2,\B(S)}$
because the Jacobian matrix of partial derivatives of $I_{s_2,\B(S)}$ has rank $1$ when evaluated at the origin, 
but it has rank 3  when evaluated at a generic point in $\N_{s_2,\B(S)}$. 

Computing a primary decomposition of $I_{s_2,\B(S)}$, we also see that the point $s_2B$ is contained in two irreducible components corresponding to the primary ideals 
$I_1=  \langle x_{32}, x_{21}x_{42} - x_{41} \rangle$ and 
$I_2= \langle x_{43}, x_{41}, x_{21} \rangle $.
The component $\mathcal{V}(I_1)$ corresponding to $I_1$ has dimension 4 (codimension 2), and 
the component $\mathcal{V}(I_2)$ corresponding to $I_2$ has dimension 3 (codimension 3).  This confirms our calculations in Example \ref{ex: 2-2 part 2} which showed  
\[
s_2B\in (\overline{C_{s_2s_3s_1}\cap \B(S)}) \cap (\overline{C_{s_2}\cap \B(S)}) \subset \X_{s_2s_3s_1}\cap \X_{s_2}
\] 
where $\dim(\X_{s_2s_3s_1}) = 4$ and $\dim(\X_{s_2})=3$ by Corollary \ref{cor: dimension}.
 
Completing similar calculations at all elements of $^MW$, we see that this Hessenberg variety has singular locus: $M(s_2B) \cup M(s_2s_1s_3B)$. We therefore recover the results of Example~\ref{ex: 2-2 part 4} using explicit calculations involving the patch ideal.
\end{example}

When applying computational techniques to study patch ideals, one needs to show that the ideals $I=I_{w, \B(S,H)}$ are radical to conclude that the schemes defined by those ideals are reduced (namely, $\N_{w, \B(S,H)}$).  
We have encountered two ways of doing this in the literature:
\begin{enumerate}
\item If the scheme $\textup{Spec } \C[U^{-}]/I$ is Gorenstein and generically reduced, then it is reduced. Thus the ideal $I$ is radical
\cite{ADGH16,IY12}.
\item If $I$ has a Gr{\"o}bner basis with  square-free lead terms, then $I$ is radical \cite[Exercise 18.9]{Eis95}.  
\end{enumerate}
When $H$ is the standard Hessenberg space, one can prove that (2) always holds.
   
\begin{Lem}  For a semisimple Hessenberg variety corresponding to the standard Hessenberg space $H_{\Delta}$ in Lie type $A$, the ideals $I_{w,\B(S)}$ for each $w\in W$ all have a square-free Gr{\"o}bner basis and are therefore radical. In particular, $I_{w,\B(S)}$ is the patch ideal of $\B(S)$ at $wB$.  
\end{Lem}
\begin{proof}[Outline of the Proof]  When $H$ is the standard Hessenberg space, the ideal $I_{w,\B(S)}$ is generated by polynomials $a_{ij}$ for $i>h(j) = j+1$ determined by the entries of the matrix $A=u^{-1}w^{-1}\cdot S$ for $u\in U^-$ a generic element.  

One then uses these polynomials to obtain simplified generators $g_{ij}$ for $i>j+1$, as in Example~\ref{ex: ideal calc}.  Each generator $g_{ij}$ has the form $g_{ij} = c_{ij} x_{ij} \pm c_{k} x_{kj}x_{ik}$ for  $j<k<i$ where $c_{ij}$ and $c_k$ could be zero.  We order the variables $x_{ij}$ of the ring $\C[U^{-}]$ by first giving preference to those furthest from diagonal and then breaking ties lexicographically with respect to the first index so $x_{k \ell}<x_{ij}$ if $i <k$.  Fix the lexicographic monomial ordering determined by this total order on the variables in $\C[U^-]$.

We now apply Buchberger's algorithm to the generators $g_{ij}$ in order to construct a Gr{\"o}bner basis of $I_{w, \B(S)}$ (see \cite[\S 7, Theorem 2]{CLO15}).  For each pair of generators $g_{ij},g_{k  \ell }$ of the patch ideal $I_{w,\B(S)}$, the $S$-polynomial $S(g_{ij},g_{k  \ell })$ has remainder zero when divided by $\{ g_{ij},g_{k  \ell }\}$. Thus the generators $\{g_{ij} : i>j+1 \}$ form a Gr{\"o}bner basis for $I_{w,\B(S)}$ that is square-free. 
\end{proof}

While we have verified that the ideals $I_{w,\B(S,H)}$ for $w\in W$ are radical for any semisimple Hessenberg variety in $GL_n(\C)/B$ with $n\leq 5$, we do not know an argument to prove this fact more generally. 

\begin{Conj}
The ideal $I_{w,\B(S,H)}$ defined above is radical for any semisimple Hessenberg variety $\B(S,H)$, and is therefore the patch ideal of $\B(S,H)$ at $wB$. 
\end{Conj}

Once we know that $I_{w,\B(S,H)}$ is the patch ideal of $\B(S,H)$ at $wB$, this ideal can be used investigate the local structure of $\B(S,H)$ at $wB$; such as whether or not $wB$ is a singular point.  Figure \ref{fig: interesting cases} lists some geometric properties of semisimple Hessenberg varieties in $GL_n(\C)/B$ for $3 \leq n \leq 4$.

\begin{figure}[h]
\begin{center}
\begin{tabular}{|l|c|c|c|c|} \hline
Hess. fun.          & Jordan Blocks of $S$    & Singular   & Irreduc.              & Equidimensional \\ \hline
(2,3,3)             & (2,1)            & Yes        & No          & Yes              \\ \hline
(2,3,4,4)           & (3,1)            & Yes        & No          & Yes              \\ \hline
(2,4,4,4)           & (3,1)            & Yes        & No          & No               \\ \hline
(3,4,4,4)           & (3,1)            & Yes        & No          & Yes              \\ \hline
(2,3,4,4)           & (2,2)            & Yes        & No          & No              \\ \hline
(2,4,4,4)           & (2,2)            & Yes        & No          & No               \\ \hline
(3,4,4,4)           & (2,2)            & Yes        & Yes         & Yes              \\ \hline
(2,3,4,4)           & (2,1,1)          & Yes        & No          & No              \\ \hline
(2,4,4,4)           & (2,1,1)          & Yes        & No          & Yes              \\ \hline
(3,4,4,4)           & (2,1,1)          & Yes        & Yes         & Yes              \\ \hline
\end{tabular}
\end{center}
\begin{caption}{Some geometric properties of semisimple Hessenberg varieties in $GL_n(\C)/B$ for $3 \leq n \leq 4$.}
\label{fig: interesting cases}
\end{caption}
\end{figure}

This table also shows that some semisimple Hessenberg varieties are irreducible and singular (unlike those Hessenberg varieties associated to the standard Hessenberg space), as the following computation illustrates.

\begin{example} 
Let $S = \textup{diag}(1,1,-1,-1)$ and $w=s_2s_1$.  The Hessenberg space corresponding to $h=(3,4,4,4)$ is $H=\fg-\fg_{-\theta}$ where $\theta$ denotes the root of maximum height in $\Phi$.  To obtain the polynomial generators defining the ideal $I_{w,\B(S,H)}$, we compute
$A= ( u^{-1}w^{-1})\cdot S $ and require that  
\[ a_{41} = 2x_{21}x_{32}x_{43} - 2x_{21}x_{42} - 2x_{31}x_{43} = 0 \]

This vanishing condition defines the ideal
\[I_{w,\B(S,H)} =   \langle 2x_{21}x_{32}x_{43} - 2x_{21}x_{42} - 2x_{31}x_{43} \rangle\] 
which is radical since it is generated by a single irreducible polynomial (as can be verified with CoCalc). The Jacobian criterion now implies that $wB=s_2s_1B$ is a singularity in the affine patch $\N_{w,\B(S,H)}$ because the Jacobian matrix of partial derivatives has rank $0$ at the origin whereas the codimension of $\N_{w,H}$ in $\N_{w,\B}$ is $1$.  As $I_{w,\B(S,H)}$ is a primary ideal, the variety is irreducible at this singularity.  Completing similar computations at each $T$-fixed point of  $\B(S,H)$ shows that the variety has only one irreducible component, and eight singular $T$-fixed points corresponding the elements of the set
$\{ s_1s_2s_3s_2,  s_3s_1s_2s_1,s_1s_2s_3, s_1s_2s_1, s_2s_3s_2, s_3s_2s_1, s_2s_1,s_2s_3\}$.

\end{example}



\bibliographystyle{alpha}

\end{document}